\documentclass[12pt]{amsart}

\usepackage{amsmath,amssymb,bm}
\usepackage{amsthm}
\usepackage{hyperref}
\usepackage[margin=1.0in]{geometry}

\numberwithin{equation}{section}

\newtheorem{prop}{Proposition}
\newtheorem{lemma}[prop]{Lemma}

\newtheorem{thm}[prop]{Theorem}
\newtheorem{cor}[prop]{Corollary}

\newtheorem{question}[prop]{Question}
\numberwithin{prop}{section}

\theoremstyle{definition}
\newtheorem{defn}[prop]{Definition}
\newtheorem{ex}[prop]{Example}
\newtheorem{rmk}[prop]{Remark}
\newtheorem{construction}[prop]{Construction}

\newcommand{\del}{\partial}
\newcommand{\dt}{\frac{\partial}{\partial t}}
\newcommand{\brs}[1]{\left| #1 \right|}

\newcommand{\gs}{\sigma}

\newcommand{\ga}{\alpha}
\newcommand{\gb}{\beta}

\newcommand{\N}{\nabla}

\newcommand{\LL}{\mathcal L}
\newcommand{\til}[1]{\widetilde{#1}}

\renewcommand{\bar}[1]{\overline{#1}}

\DeclareMathOperator{\Rc}{Rc}

\DeclareMathOperator{\tr}{tr}
\DeclareMathOperator{\Ker}{Ker}

\renewcommand{\Im}{\mathop{\mathrm{Im}}}
\renewcommand{\Re}{\mathop{\mathrm{Re}}}
\newcommand{\la}{\langle}
\newcommand{\ra}{\rangle}
\newcommand{\C}{\mathbb C}
\newcommand{\R}{\mathbb R}
\newcommand{\Z}{\mathbb Z}
\newcommand{\mc}{\mathcal}
\newcommand{\mf}{\mathfrak}
\newcommand{\fp}{\varphi}

\begin{document}

\author{Vestislav Apostolov}
\address{Vestislav Apostolov\\ D\'epartment de mathe\'ematiques\\
	Universit\'e du Qu\'ebec \`a Montr\'eal\\
	Case postale 8888, succursale centre-ville
	Montr\'eal (Qu\'ebec) H3C 3P8 \\ and Laboratoire Jean Leray   \\ Universit\'e de Nantes\\
2, Rue de la Houssinière - BP 92208
F-44322 Nantes\\ France }
\email{\href{mailto:apostolov.vestislav@uqam.ca}{apostolov.vestislav@uqam.ca}}

\author{Jeffrey Streets}
\address{Jeffrey Streets\\Rowland Hall\\
	University of California\\
	Irvine, CA 92617}
\email{\href{mailto:jstreets@uci.edu}{jstreets@uci.edu}}

\author{Yury Ustinovskiy}
\address{Yury Ustinovskiy\\Lehigh University\\
	Chandler-Ullmann Hall\\
	17 Memorial Drive East\\
	Bethlehem, PA, 18015} 
\email{\href{mailto:yuu221@lehigh.edu}{yuu221@lehigh.edu}}

\title{Variational structure and uniqueness of generalized K\"ahler-Ricci solitons}

\begin{abstract} Under broad hypotheses we derive a scalar reduction of the generalized K\"ahler-Ricci soliton system.  We realize solutions as critical points of a functional analogous to the classical Aubin energy defined on the orbit of a natural Hamiltonian action of diffeomorphisms, thought of as a generalized K\"ahler class.  This functional is convex on a large set of paths in this space, and using this we show rigidity of solitons in their generalized K\"ahler class.  As an application we prove uniqueness of the generalized K\"ahler-Ricci solitons on Hopf surfaces constructed in \cite{st-us-19}, finishing the classification in complex dimension $2$.
\end{abstract}

\date{\today}

\maketitle

\section{Introduction}

Given a smooth manifold $M$, a Riemannian metric $g$, closed three-form $H$, and smooth function $f$ on $M$ define a \emph{generalized Ricci soliton} if
\begin{align*}
	\Rc-\frac{1}{4}H^2+\nabla^2f=&\ \lambda g,\\
	d^* H+i_{\nabla f}H=&\ 0,
\end{align*}
where $\lambda \in \R$ is a constant.  The generalized Ricci soliton is called \emph{shrinking, steady} or \emph{expanding} according to $\lambda>0$, $\lambda=0$ or $\lambda<0$. These structures are natural generalizations of Einstein metrics, and arise from considerations in complex geometry \cite{Ivanovstring,PCFReg}, Weyl geometry \cite{GauduchonIvanov}, and mathematical physics \cite{Friedanetal}.  Also, generalized Ricci solitons are critical points of functionals generalizing Perelman's entropies~\cite{OSW, Streetsexpent}. 

In this work we will focus on solutions which are also \emph{generalized K\"ahler}, and are thus called \emph{generalized K\"ahler-Ricci solitons} (GKRS).  A generalized K\"ahler manifold (\cite{GHR, gu-14}) is defined by $(M, g, I, J)$ where $I$ and $J$ are integrable complex structures, $g$ is compatible with both $I$ and $J$, and furthermore 
	\[
	d^c_I\omega_I=H = -d^c_J\omega_J, \qquad dH=0,
	\]
where $\omega_I=gI$ and $\omega_J=gJ$.   As a special case, we can take $(M, g, I)$ to be K\"ahler, $J=\pm I$, and $H=0$,  so that GKRS reduce in this case  to \emph{gradient K\"ahler-Ricci solitons} (KRS). Here a fundamental point is that a compact steady or expanding K\"ahler-Ricci soliton automatically has $f$ constant and so $g$ is K\"ahler-Einstein with zero or negative scalar curvature.  In this case, Calabi showed \cite{Calabi} using the maximum principle these metrics are unique in a fixed K\"ahler class. In the compact, non-Einstein case, shrinking KRS  have been extensively studied in recent years~\cite{TZ1, TZ2, BN, HL} and a general uniqueness result  up to the action of the group of automorphisms is established in~\cite{TZ2,BN}.

A fundamental example of non-K\"ahler generalized K\"ahler structure  \cite[Example\,1.23]{gu-14}  is given by a compact even dimensional semisimple Lie group  with bi-invariant  Riemannian metric $g$, and $I = J_L$, $J = J_R$ compatible left-invariant and right-invariant complex structures respectively.  The associated Bismut connections are flat, and the triples $(g, I, J)$ define steady GKRS with $f$ constant.  In complex dimension $2$, this construction yields the Hopf/Boothby metric on \emph{standard} Hopf surfaces.  In \cite{st-us-19}, the second and third authors constructed steady GKRS on all \emph{diagonal} Hopf surfaces (see Definition \ref{def:hopf}), which have $f$ non-constant aside from the standard case.

The goal of this paper is to address general uniqueness phenomena for GKRS, and in particular to finish the classification of GKRS in complex dimension $2$. To this end, we recall that the second author has shown in \cite{st-19-soliton} that  a GKRS on a  compact complex surface is either a KRS, or else must be a \emph{steady} GKRS defined on a Hopf surface. We will thus focus  in this paper on the uniqueness properties in steady case.  Whereas K\"ahler Calabi-Yau metrics are unique in a fixed K\"ahler class, it turns out that the non-K\"ahler situation has some reminiscences with the basic structural theory  of compact shrinking KRS on Fano manifolds. Indeed,  the group of complex automorphisms of $(M, I)$ is not trivial in general, and the theory of shrinking KRS will guide the discussion to follow.

To begin,  it is shown in \cite{st-us-20-GH} that every (steady) GKRS comes equipped with vector fields
\begin{align*}
X_I = \frac{1}{2} I \left( \theta_I^{\sharp} - \N f \right), \qquad X_J = \frac{1}{2} J \left( \theta_J^{\sharp} - \N f \right)
\end{align*}
such that $X_I$ is $I$-holomorphic and $X_J$ is $J$-holomorphic (cf. Proposition \ref{prop:soliton_equiv}).  Here $\theta_I$ is the Lee form of the Hermitian structure $(g, I)$, likewise $\theta_J$, and {$\sharp$ is the isomorphism $g^{-1} : T^*M \cong TM$.  Furthermore the GKRS system is expressed as
\begin{align*}
(\rho_I^B)_I^{1,1} + \LL_{I X_I} \omega_I = 0,
\end{align*}
where $\rho_I^B$ is the Bismut-Ricci form of the Hermitian structure $(g, I)$. Our first main goal is to derive a scalar reduction of the GKRS system under some broad structural hypotheses on the underlying GK structure.   Given a GK structure $(M, g, I, J)$, on the loci where $I \pm J$ are invertible there are associated symplectic forms
\begin{align*}
F_{\pm} = -2 g(I \pm J)^{-1}.
\end{align*}
If one of these is defined everywhere on $M$ the GK structure is called \emph{symplectic type}.  We show in Proposition \ref{pr:symp_soliton} below that a GKRS of symplectic type is automatically K\"ahler, Calabi-Yau.  Thus outside of the K\"ahler setting neither $F_{\pm}$ can be globally defined, but we will assume instead that \emph{both} are defined on open dense proper subsets $U_{\pm}$, and such GK structures are called \emph{log-nondegenerate}.  This is the case for any GK structure on a compact 4-manifold with odd first Betti number, such that $I$ and $J$ induce the same orientation (known as GK structures of \emph{even type} ~\cite{ap-ga-gr-99}), and in particular for the GKRS on Hopf surfaces constructed in \cite{st-us-19}, where $U_{\pm}$ are each the complement of disjoint elliptic curves.  In the log-nondegenerate setting we show that $X_I,X_J$ commute, preserve $I,J$ and $g$, and furthermore that the two vector fields $X_I \pm X_J$ are $F_{\pm}$-Hamiltonian with normalized Hamiltonian potentials $\psi_{\pm}$ (see Proposition~\ref{pr:XI_XJ_properties}). We show that the gradient steady GKRS equations imply the (in fact globally defined) scalar equation
\begin{align} \label{f:scalarGKRS}
e^{\psi_+} F_+^n = e^{\psi_-} F_-^n.
\end{align}
Conversely, a log-nondegenerate GK manifold together with vector fields $X_I$, $X_J$ satisfying the symmetries above and such that (\ref{f:scalarGKRS}) holds, is a GKRS (see Proposition \ref{prop:gk_soliton_suff} for the precise statement).

The scalar reduction (\ref{f:scalarGKRS}) is useful in understanding uniqueness of GKRS in a \emph{generalized K\"ahler class}, a notion we now recall.  First, every generalized K\"ahler manifold comes equipped with a real Poisson tensor \cite{ap-ga-gr-99, hi-06}
\begin{align*}
\gs = \frac{1}{2} [I, J] g^{-1}.
\end{align*}
Given a one-parameter family of smooth functions $\fp$ on $M$, we let $X_{\fp} = - \gs(d \fp)$ be the associated family of $\gs$-Hamiltonian vector fields, with associated family of diffeomorphisms $\Phi_t$.  In the log-nondegenerate setting, we can describe a GK structure equivalently in terms of the data $(I, J, \gs)$, and then $\Phi_t$ acts on GK structures of this type to give the one-parameter family $(\Phi_t^* I, J, \gs)$ for $t$ small enough.  This is an example of the `flow construction' in GK geometry \cite{ap-ga-gr-99, bi-gu-za-18, gu-10}.  The orbit $\mc M$ of the flow construction is  a formal Fr\'echet manifold which is modelled on 
$C^{\infty}(M,\R)/\R$~--- the space of smooth functions modulo additive real constants~--- and is a direct generalization of the idea of K\"ahler class in K\"ahler geometry. We refer to the orbits $\mc M$ henceforth as \emph{generalized K\"ahler classes}.  

Following ideas in \cite{ap-st-17}, we define a functional on the $G$-\emph{invariant} generalized K\"ahler class of a log-nondegenerate GK structure, where $G$ is the compact torus in the isometry group of $(M,g)$  generated  by the flows of vector fields $X_I$, $X_J$ satisfying the structural conditions above. More precisely,  we define a closed $1$-form on the associated $G$-invariant generalized K\"ahler class $\mc M^G$ via
	\[
	\pmb\nu_{X_I, X_J}(\fp):=\int_M \fp \left(
		e^{\psi_+}\frac{F_+^n}{n!}-e^{\psi_-}\frac{F_-^n}{n!}
	\right).
	\]
	The above $1$-form is somewhat reminiscent to the definition of the $\pmb{J}$-functional in K\"ahler geometry, and its weighted versions considered in~\cite{TZ1}, which are introduced through the $1$-form  
	\[\pmb{\mu}_X(\fp) :=\int_M \fp \left(e^{\psi_{\omega_0}} \frac{\omega_0^n}{n!}- e^{\psi_{\omega}} \frac{\omega^{n}}{n!}\right), \]
	where $\psi_{\omega}$ is a suitably normalized $\omega$-Hamiltonian potential of a given Killing vector field $X$,  defined on the space of $X$-invariant K\"ahler metrics $(g, J, \omega)$ in a given  K\"ahler class. We however stress here that  $\pmb\nu_{X_I, X_J}$ above depends essentially on the log-nondegenerate assumption of the holomorphic Poisson tensor, and thus has no precise analogue in the theory of compact KRS.
	
Integration of $\pmb \nu$ along paths yields a well-defined functional $\pmb J_{X_I, X_J}$ on the universal cover of the generalized K\"ahler class.  Remarkably, this functional is convex along families of Hamiltonian diffemorphisms generated by time-independent potentials, and this implies a rigidity result for steady GKRS in all dimensions (cf. Theorem~\ref{thm:rigid_soltions} below):

\begin{thm}[Rigidity of solitons]\label{thm:rigidite_soltions}
	Let $(M,g_t,I_t,J)$, $t \in [0, \tau)$ be a one-parameter family of compact log-nondegenerate GKRS lying in $\mc M^G$, where $G$ is the torus generated by the vector fields $X_I, X_J$ associated to $(g_0, I_0, J)$.  Then $(M,g_t,I_t,J)=(M,g_0,I_0,J)$ for all $t$.
\end{thm}

Our main application of this new theory is to finish the classification of compact generalized K\"ahler-Ricci solitons on complex surfaces initiated in \cite{st-19-soliton, st-us-19}. One of the main results in~\cite{st-us-19} states the existence of a GKRS on a diagonal Hopf surface $(M, I_{\ga\gb})$ (cf. Theorem \ref{thm:hopf_soliton_existence}).  By construction, these solitons are invariant under the maximal compact subgroup $K \subset \mathrm{Aut}(M,I_{\alpha\beta}, \pmb{\pi})$, where $\pmb{\pi}$ is the standard holomorphic Poisson tensor (see (\ref{eq:hopf_poisson})), and hence define solitons on certain secondary Hopf surfaces, which are finite quotients of primary Hopf surfaces.  Thus, to complete the classification we must determine precisely which secondary Hopf surfaces admit solutions, and furthermore to establish their uniqueness, up to {complex automorphisms}.  By an explicit analysis of the automorphism groups of Hopf surfaces, we show in Theorem \ref{thm:hopf_existence_even} that a secondary Hopf surface admits an even-type GK structure if and only if it is a quotient of a diagonal Hopf surface by a certain \emph{cyclic} subgroup $\Gamma_{k,\ell}$ of the automorphism group.  As these cyclic deck groups are contained in the {maximal compact subgroup $K\subset \mathrm{Aut}(M,I_{\alpha\beta}, \pmb{\pi})$ defined above}, the solitons of Theorem \ref{thm:hopf_soliton_existence} descend to the quotients, settling the existence question. To establish the uniqueness, the key issue is to establish maximal $K$-symmetry of a given GKRS, and this is achieved in Section~\ref{sec:hopf} via the rigidity of Theorem \ref{thm:rigidite_soltions}.  The final classification statement {reads as} follows:

\begin{cor} \label{c:classification1}
	Let $(M^4,g,I,J)$ be a compact gradient steady GKRS. Then precisely one of the following holds
	\begin{enumerate}
		\item $(M^4,I)$ is K\"ahler, $g$ is a Calabi-Yau metric, i.e., $\mathrm{Rc}=0$, and $I=\pm J$.
		\item $(M^4,g,I,J)$ is an odd-type GK structure, with $(M^4,I)$ biholomorphic to a (possibly trivial) quotient of a diagonal Hopf surface $(\til M,I_{\alpha\beta})$ by $\Gamma_{k,\ell}\simeq \Z_{\ell}$.
		Up to the action of $\mathrm{Aut}(M,I)$ and scaling, the metric $g$ and the complex structure $J$ are given by $g^s$ and $\pm J_{\mathrm{odd}}^s$ of Theorem~\ref{thm:hopf_soliton_existence}.
		\item $(M^4,g,I,J)$ is an even-type GK structure, with $(M^4,I)$ biholomorphic to a (possibly trivial) quotient of a diagonal Hopf surface $(\til M,I_{\alpha\beta})$ by $\Gamma_{k,\ell}\simeq \Z_{\ell}$.
		Up to the action of $\mathrm{Aut}(M,I)$ and scaling, the metric $g$ and the complex structure $J$ are given by $g^s$ and $\pm J_{\mathrm{even}}^s$ of Theorem~\ref{thm:hopf_soliton_existence}.
	\end{enumerate}
\end{cor}

\begin{rmk} \label{r:flowrmk} As Corollary \ref{c:classification1} has now given a complete description of the existence and uniqueness of GKRS on complex surfaces, it is natural to expect that these are global attractors for GKRF.  That is, on Hopf surfaces admitting a GKRS, one expects (cf. \cite{st-geom}) that the GKRF (and even more generally pluriclosed flow) exists globally and converges to the unique (up to scale and automorphism) soliton.    This was shown for certain invariant initial data (\cite{st-us-19} Theorem 1.1 (3)).  Recently, in \cite{ga-jo-st-21}, it was shown that on a \emph{standard} Hopf surface $(M, I)$, the pluriclosed flow (and hence the GKRF) always exists globally and converges to a conformally flat Hopf metric $g^{s}$ (which is also a GKRS).
\end{rmk}

\section{Background}\label{sec:bg}
%
%
%
%

%
%
%
%
In this section we collect the necessary background on generalized K\"ahler geometry. We follow the \emph{bihermitian} interpretation of the generalized K\"ahler geometry instead of the more fundamental formulation in the language of \emph{generalized geometry}. For a more detailed introduction to the subject see, e.g.,~\cite{gu-14}.

\begin{defn}[Generalized K\"ahler structures] \label{def:gk}
	A manifold $M^{2n}$ with a Riemannian metric $g$ and two compatible complex structures $I$ and $J$ is called \emph{generalized K\"ahler} (GK) if
	\[
	d^c_I\omega_I= H = -d^c_J\omega_J\qquad\mbox{and}\qquad d H=0,
	\]
	where $\omega_I=gI$ and $\omega_J=gJ$.  A GK manifold $(M^{4n},g,I,J)$ of dimension divisible by 4 is \emph{even}, if $I$ and $J$ induce the same orientation on $TM$, otherwise it is called \emph{odd}~\cite[Remark\,1.13]{gu-14}.
\end{defn}

We call the 3-form $H:=d^c_{I}\omega_I$ the \emph{torsion} of $(M^{2n},g,I,J)$ and denote by $\theta_I,\theta_J\in \Gamma(\Lambda^{1}(M,\R))$ the \emph{Lee forms} of the underlying Hermitian structures $(M^{2n},g,I)$ and $(M^{2n},g,J)$ respectively, characterized by the identities
\begin{equation*}
d\omega_I^{n-1}=\theta_I\wedge\omega_I^{n-1},\qquad d\omega_J^{n-1}=\theta_J\wedge\omega_J^{n-1}.
\end{equation*}

\begin{defn}[Log-symplectic GK structures]
	Let $(M^{2n},g,I,J)$ be a GK manifold. If the operator $I+J$ is globally invertible, then $M$ admits a distinguished symplectic form
	\[
	F_+=-2g(I+J)^{-1}.
	\]	
	In this case we will say that a GK manifold  $(M^{2n},g,I,J)$ is \emph{symplectic-type}. If instead the operator $I-J$ is globally invertible, we will still say that $(M^{2n},g,I,J)$ is \emph{symplectic-type}, tacitly assuming that the underlying symplectic form is
	\[
	F_-=-2g(I-J)^{-1}.
	\]
	If $I+J$ (or $I-J$) is invertible on an open dense set
	\[
	U_+=\{x\in M\ |\ (I+J)_x \mbox{ is invertible}\}
	\]
	(resp. $U_-=\{x\in M\ |\ (I-J)_x \mbox{ is invertible}\}$)
	we will say that $(M,g,I,J)$ is \emph{log-symplectic}.
\end{defn}

Associated with any GK manifold $(M,g,I,J)$ there is a real Poisson tensor:
\begin{equation*}
\sigma=\frac{1}{2}[I,J]g^{-1},
\end{equation*}
which is a common real part of complex-valued $I$- (resp.\,$J$-) holomorphic type (2,0) Poisson tensors
\begin{equation}\label{eq:poisson_holo}
\sigma_I=\sigma-\sqrt{-1}I\sigma\quad (\mbox{resp. }\sigma_J=\sigma-\sqrt{-1}J\sigma),
\end{equation}
see \cite{ap-ga-gr-99, po-97, hi-06}.  Assuming that $\sigma$ is invertible at a point in $M$, then $\sigma_I^{-1}$ is locally a holomorphic symplectic form, so $n$ is necessarily even, and the tensor $\sigma_I^{n/2}$ provides a nonzero $I$-holomorphic section of the anticanonical bundle $-K_{M,I}$. Thus the zero locus
\[
\{\sigma_I^{n/2}=0\}=\{x\in M\ |\ [I,J]=(I-J)(I+J) \mbox{ is not invertible} \}=M\backslash (U_+\cap U_-)
\]
is an $I$-analytic subset of real codimension 2. Interchanging the roles of $I$ and $J$ we also conclude that $M\backslash (U_+\cap U_-)$ is $J$-analytic.

\begin{defn}[Log-nondegenerate GK structures]
	An even GK manifold $(M^{2n},g,I,J)$ is \emph{nondegenerate} if the underlying Poisson tensor $\sigma$ is globally invertible. If $\sigma$ is invertible at one point, then by the above observation, $\sigma$ is invertible on a complement of an analytic set
	\[
	Z=\{\sigma_I^{n/2}=0\}=M\backslash (U_+\cap U_-).
	\]
	In the latter case we say that $(M,g,I,J)$ is \emph{log-nondegenerate}.
\end{defn}
\begin{rmk}
Let us make several important observations about log-symplectic and log-nondegenerate GK structures:
\begin{enumerate}
	\item Since the complement of $U_+\cap U_-$ in $M$ is a simultaneously $I$ and $J$ analytic set, either $U_+\cap U_-$ is open and dense, or it is empty.
	\item If $(M,g,I,J)$ is log-nondegenerate, then on the open dense subset $U_+\cap U_-$ the tensor $[I,J]$ is invertible, so we can define three symplectic forms $\sigma^{-1}$, $\sigma^{-1}I$, $\sigma^{-1}J$.
	\item The symplectic forms
	\[
	F_{\pm}=-2g(I\pm J)^{-1}
	\]
	defined on $U_{\pm}$ satisfy the identity  $F_{\pm}=\sigma^{-1}J\mp\sigma^{-1}I$ on $U_+\cap U_-$.
	\item The Riemannian metric $g$ can be recovered from the symplectic forms $F_{\pm}$ as
	\[
	g=-(F_+ J)^{\mathrm{sym}}=(F_- J)^{\mathrm{sym}},
	\]
	where $\mathrm{sym}$ denotes the symmetric component of a tensor in $T^*M^{\otimes 2}$ (see~\cite[\S 2.3]{ap-st-17}).
\end{enumerate}
\end{rmk}

In what follows we extend a complex structure $I\colon TM\to TM$ to an operator on $T^*M\to T^*M$ using the metric contractions $g\colon TM\to T^*M$ and $g^{-1}\colon T^*M\to TM$ so that for $\alpha\in T^*M$
\[
(I\alpha)(X):=(gIg^{-1}\alpha)(X)=-\alpha(IX).
\]

On the locus where $I+J$ and $I-J$ are both invertible, there are important identities relating $\theta_I$ and $\theta_J$ with $\det(I+J)$ and $\det(I-J)$.

\begin{lemma}[{\cite[Lemma\,3.8]{ap-st-17}}]\label{lm:lee_identity}
	Let $(M,g,I,J)$ be a log-nondegenerate GK structure. Then on $U_{\pm}$ we have
	\begin{equation}\label{eq:lee_identity}
	I\theta_I\pm J\theta_J=\frac{1}{2}(I\pm J)d\log\det(I\pm J).
	\end{equation}
	Furthermore, $\theta_I^\sharp-\theta_J^\sharp=\sigma d\Phi$, where $\Phi:=\frac{1}{2}\log\frac{\det(I-J)}{\det(I+J)}$.
\end{lemma}

\begin{rmk}
	If $\dim_\R M=4$ and GK structure $(M^4,g,I,J)$ is even, then either $(M^4,g,I,J)$ is K\"ahler with $I=\pm J$, or $M$ is log-nondegenerate and $U_+\cap U_-\subset M$ is an open dense subset with the \emph{non-intersecting} complements $Z_{\pm}=M^4\backslash U_{\pm}$ being analytic sets $Z_{\pm}=\{x\in M^4\ |\ (I\pm J)_x=0\}$. Thus any non-K\"ahler even GK structure on $M^4$ defines a \emph{disconnected} canonical divisor $Z_+\cup Z_-$.
\end{rmk}

%
%
%
%

The fundamental examples of GK structures on non-K\"ahler surfaces are given by Hopf surfaces.

\begin{defn}\label{def:hopf}
	A \emph{primary} Hopf surface is a quotient of $\C^2\backslash\{\mathbf{0}\}$ by a holomorphic contraction $A\colon\C^2\backslash\{\mathbf{0}\}\to \C^2\backslash\{\mathbf{0}\}$:
	\[
	(M,I)=(\C^2\backslash\{\mathbf{0}\})/\la A\ra
	\]
	By a result of Kodaira, in appropriate coordinates $(z_1,z_2)$ on $\C^2$, the contraction $A$ is given by
	\[
	A\colon (z_1,z_2)\mapsto (\alpha z_1+\lambda z_2^q,\beta z_2),
	\]
	where $\alpha,\beta\in\C$ satisfy $1<|\alpha|\leq|\beta|<1$ and either $\lambda=0$ or $\alpha=\beta^q$. If $\lambda=0$, and $A$ is given by
	\begin{equation}\label{eq:hopf_def_A}
		A\colon (z_1,z_2)\mapsto (\alpha z_1,\beta z_2),
	\end{equation}
	the corresponding Hopf surface $(M,I_{\ga\gb})$ is called \emph{diagonal}, and in the special case $\brs{\ga} = \brs{\gb}$ we call it \emph{standard}.  One can easily see that in all cases $M\simeq S^3\times S^1$ as a smooth manifold.
	
	A free finite holomorphic quotient of a primary Hopf surface is called a \emph{secondary} Hopf surface. All such quotients were classified by Kato in~\cite{ka-89,ka-89-er}.  In this paper we will encounter only \emph{cyclic quotients} of diagonal Hopf surfaces $(M,I_{\alpha\beta})$ by a subgroup
	\[
	\Gamma_{k,\ell}\subset \mathrm{Aut}(M,I_{\alpha\beta}),\quad \Gamma_{k,\ell}\simeq \Z_\ell
	\]
	generated by the multiplication by primitive $\ell^{\mathrm{th}}$ roots of unity:
	\begin{equation}\label{Gamma}
	(z_1,z_2)\mapsto (\exp{\tfrac{2\pi \sqrt{-1}k}{\ell}}z_1, \exp{\tfrac{2\pi \sqrt{-1}}{\ell}}z_2)
	\end{equation}
	where $\ell>k\geq 0$ and $(k,\ell)=1$.
\end{defn}

\begin{rmk}
It is proved in~\cite[Theorem\,3]{ap-gu-07} that if a Hopf surface $(M,I)$ is a part of an \emph{odd} GK structure $(M,g,I,J)$, then $(M,I)$ is biholomorphic to a quotient of the diagonal Hopf surface $(M,I_{\alpha\beta})$ by a cyclic group $\Gamma_{k,l}\simeq \Z_\ell$. Conversely, any such Hopf surface is a part of an odd GK structure.  It is also observed in~\cite[Prop.\,3.2]{st-us-19} that any diagonal Hopf surface is a part of an explicit even GK structure (see  also \cite[Example\,1.21]{gu-14} for the case $|\alpha|=|\beta|$).
\begin{ex}[{Generalized K\"ahler structure on the diagonal Hopf surfaces~\cite[Ex.\,3.10]{st-us-20-GH}}]\label{ex:hopf_diagonal_gk}
	Let $(M,I_{\alpha\beta})$ be a diagonal Hopf surface:
	\[
	M=(\C^2_{z_1,z_2}\backslash\{\mathbf 0\})/\la A\ra.
	\]
	In~\cite{st-us-19} and~\cite{st-us-20-GH} we have described a family of GK structures on $M$ determined by one scalar function.
	
	Let $\C^2\to(\C^*)^2/\la A\ra\subset \C^2\backslash\{\mathbf 0\}/\la A\ra=M$ be the universal cover of the nondegenerate part of the Hopf surface. Choose $w_i=\log z_i$ to be the coordinates in $\C^2$, $w_i=x_i+\sqrt{-1}y_i$ and denote $a=\log|\alpha|$, $b=\log|\beta|$.  Given a function $p\colon \R\to (-1,1)$ of a real argument $2(\frac{b}{a} x_1-x_2)$, we can define a second complex structure $J$ on $\C^2$ via a complex-valued symplectic form (in~\cite{st-us-19} we used a function $k=(p+1)/2$ instead of $p$):
	\[
	\begin{split}
	\Omega_J&=\left(dw_1-\frac{a}{b}d\bar w_2\right)\wedge \left(\frac{b(1+p)}{2a}d\bar{w_1}+\frac{1-p}{2}dw_2\right).
	\end{split}
	\]
	The form $\Omega_J$ is decomposable, closed and nondegenerate thus it uniquely defines an integrable complex structure $J$ such that $dw_1-\frac{a}{b}d\bar w_2$ and $\frac{b(1+p)}{2a}d\bar{w_1}+\frac{1-p}{2}dw_2$ span $\Lambda^{1,0}_I$. One can directly check that with
	\[
	g=\Re\left(\frac{b(1+p)}{2a}dw_1\otimes d\bar{w_1}+\frac{a(1-p)}{2b}dw_2\otimes d\bar{w_2}\right).
	\]
	the triple $(g,I_{\alpha\beta},J)$ defines a GK structure on $\C^2_{w_1,w_2}$ (here by abuse of notation we denote $I_{\alpha\beta}$ the standard complex structure on $\C^2_{w_1,w_2}$). This structure descends to $(\C^*)^2_{z_1,z_2}/\la A\ra$, and under certain assumptions on the asymptotic of $p(x)$ as $x\to\pm\infty$, extends to a smooth GK structure $(M,g,I_{\alpha\beta},J)$ on the Hopf surface.
\end{ex}

In Section~\ref{sec:hopf} below we study the existence of GK structures on secondary Hopf surfaces and show that,  similarly to the odd case, secondary Hopf surfaces admit even GK structures only in the case of cyclic quotients \eqref{Gamma}.
\end{rmk}


%
%
%
%
\section{Hamiltonian flow construction}\label{sec:flow_construction}

Let $(M,\omega, I)$ be a K\"ahler manifold. A remarkable feature of K\"ahler geometry is the notion of a \emph{K\"ahler class}~--- for any smooth function $\phi\in C^\infty(M,\R)$ we can define a new (1,1)-form
\[
\omega_\phi:=\omega+dd^c_I\phi,
\]
and as long as $\omega_\phi$ is positive definite, $(M,\omega_\phi, I)$ is a K\"ahler manifold with $[\omega_\phi]=[\omega]\in H^{1,1}(M,\R)$. In particular, for any $\fp\in C^\infty(M,\R)$ there is an infinitesimal deformation $\omega_t$ of $\omega$, such that 
\[
\dt \omega_t=dd^c_I\fp.
\]
In this section we review a similar construction in generalized K\"ahler geometry. This construction was first presented in the context of bihermitian structures on 4-dimensional manifolds in~\cite[\S 4.2]{ap-ga-gr-99}, where it is attributed to Joyce (see also Hitchin's GK structures on del Pezzo surfaces~\cite{hi-07}). In~\cite{li-ro-vu-za-07} the authors presented a version of the construction in higher dimensions. Gualtieri in \cite[\S 7]{gu-10} defined the Hamiltonian flow construction on a symplectic GK manifold. The presentation below follows ~\cite{gi-st-20}.  We start by fixing some terminology.

\begin{defn}\label{def:hamiltonian_vf}
	Let $(M,\sigma)$ be a smooth manifold with a real Poisson tensor $\sigma\in \Gamma(\Lambda^2(TM))$. We will say that a vector field $X\in \Gamma(TM)$ is $\sigma$-\emph{Hamiltonian}, if there exists a function $\fp\in C^\infty(M,\R)$ such that
	\[
	X=-\sigma(d\fp).
	\]
	To specify the role of the potential $\fp$ explicitly, we will often write $X=X_{\fp}$.  We will say that a vector field $X$ \emph{preserves $\sigma$} if $\mc L_{X}\sigma=0$.
\end{defn}

	Clearly any $\sigma$-Hamiltonian vector field $X_\fp$ preserves $\sigma$.  The converse statement is true locally if $\sigma$ is invertible, as one can take $\fp$ to be the local antiderivative of the closed 1-form $-\sigma^{-1}(X)$. However, if $\sigma$ is only generically invertible and $X$ preserves $\sigma$, there might not be a smooth Hamiltonian potential $\fp$ for $X$ even locally. For instance, the vector field $X=\del_y$ on $\R^2_{x,y}$ preserves the Poisson structure
	\[
	\sigma=x\del_x\wedge\del_y,
	\]
	yet there is no local Hamiltonian potential in a neighbourhood of $\{x=0\}\subset\R^2$. Of course, we could extend our notion of $\sigma$-Hamiltonian potentials, and consider a function $\fp=-\log|x|$ as a singular Hamiltonian potential for $X$.

	The Poisson structure $\sigma$ induces a bracket $\{\cdot,\cdot\}_\sigma$ on $C^\infty(M,\R)/\R$,  defined as
	\[
	\{\fp,\psi\}_\sigma:=\sigma(d\fp,d\psi).
	\]
	This pairing turns $C^\infty(M,\R)/\R$ into an infinite-dimensional Lie algebra. Given a Hamiltonian vector field $X_\fp$ and a function $\psi$, we have
	\[
	X_{\fp}\cdot \psi=d\psi(X_{\fp})=-\sigma(d\fp,d\psi)=-\{\fp,\psi\}_{\sigma}.
	\]
	Now, if $X_{\fp}$ and $X_{\psi}$ are two Hamiltonian vector fields, then
	\[
	[X_{\fp},X_{\psi}]=-\mc L_{X_\fp}(\sigma(d\psi))=\sigma(\{\fp,\psi\}_\sigma)=-X_{\{\fp,\psi\}_\sigma},
	\]
	so that the Lie algebra $(C^\infty(M,\R)/\R, \{\cdot,\cdot\}_\sigma)$ is anti-isomorphic to the Lie algebra of Hamiltonian vector fields under the Lie bracket.

\begin{construction}[Flow construction of GK structures \cite{bi-gu-za-18,gi-st-20}]\label{cstr:flow}
	Let $(M,g,I,J)$ be a GK manifold. 
	Denote by $\fp_t\in C^\infty(M,\R)$ a one-parameter family of smooth functions on $M$. Consider 
	\[
	X_{\fp_t}:=-\sigma(d\fp_t)
	\]
	the $\sigma$-Hamiltonian vector field induced by the potential family $\fp_t$. Let $(g_t,I_t,J)$ be a one-parameter family of tensors solving
	\begin{equation} \label{f:flowconst}
	\left\{\begin{split}
	\dt g&=(dd^c_I\fp_t J)^{\mathrm{sym}},\\
	\dt I&=\mc L_{X_{\fp_t}}I,\\
	\dt J&=0.
	\end{split}
	\right.
	\end{equation}
	It is proved in~\cite[Theorem\,1.1; Prop.\,3.6]{gi-st-20} that such family defines a one-parameter family of generalized K\"ahler structures on $M$, as long as $g_t$ stays positive definite. The variation of $I$ can be equivalently expressed as
	\[
	\dt I=\sigma(dd^c_I\fp_t).
	\]
	Since in the family $(g_t,I_t,J)$ the complex structure $I_t$ flows by a family of $\sigma$-Hamiltonian diffeomorphisms, we call $(g_t,I_t,J)$ the \emph{flow construction} induced by a family of potentials~$\fp_t$.
			
	Along the flow construction we have the induced variations of the symplectic forms (whenever they are defined) and see that the Poisson tensor is fixed:
	\begin{equation}\label{eq:flow_evolution}
	\begin{split}
	\dt F_+&=-dd^c_I\fp_t,\\
	\dt F_-&=dd^c_I\fp_t,\\
	\dt \sigma &= 0.
	\end{split}
	\end{equation}
\end{construction}

The flow construction can be also defined for certain \emph{singular} functions $\fp$ as long as $dd^c_I\fp_t$ extends to a well-defined smooth 2-form on $M$ (see e.g.,~\cite{hi-07}). For example, the flow construction yields the trivial deformation if $(M,g,I,J)$ is log-nondegenerate and $\fp\in C^\infty(U_{\pm},\R)$ is an $I$-pluriharmonic function with logarithmic poles along the zero locus of $\sigma$. In this case, the vector field $X_{\fp}$ preserves $\sigma$ and is $I$-holomorphic, so that $I,J$ and $g=2\sigma^{-1}[I,J]$ are all stationary.

Underlying the flow construction, there is an infinite-dimensional Lie group $\mc H_\sigma$ with the Lie algebra $\mf{h}_\sigma=\left(C^\infty(M,\R)/\R,\ \{\cdot,\cdot\}_\sigma\right)$, where $\{\cdot,\cdot\}_{\sigma}$ is the Poisson bracket induced by $\sigma$. There is a natural homomorphism 
\[
\mc H_{\sigma}\to \mathrm{Diff}(M)
\]
where both groups act on $M$ from the right, so that along the flow construction $I_t$ is pulled back by a $\sigma$-Hamiltonian diffeomorphism. If $(M,g,I,J)$ is log-nondegenerate so that $\sigma$ is generically invertible, the above homomorphism is injective.

Thus we can think of the flow construction as the \emph{local} action of the group $\mc H_\sigma$ on the space $\mc{GK}(M)$ of generalized K\"ahler structures on $M$:
\[
\mc{GK}(M)\times \mc H_\sigma\to \mc{GK}(M).
\]
We denote by $\mc M=\mc M_{m_0}$ the orbit of a given structure $m_0=(M,g_0,I_0,J)$ under this action. Since the infinitesimal action induced by the flow potential $\fp_t$ is trivial if and only if $dd^c_I\fp_t=0$, and on a compact manifold $\Ker(dd^c_I\big|_{C^\infty(M,\R)})=\R$, we conclude that for any $m\in \mc M_{m_0}$ there is an identification
\[
T_{m}\mc M\simeq C^\infty(M,\R)/\R.
\]
Furthermore, with this identification any fixed smooth function corresponds to a right-invariant vector field on $\mc H_\sigma$, which descends to the corresponding fundamental vector field on $\mc M$ via the above action.

If there is a compact Lie group $G\subset \mathrm{Aut}^0(M,J)$ acting on $M$, then we can similarly define the $G$-equivariant flow construction by considering the space of $G$-equivariant flow potentials
\[
C^\infty(M,\R)^G=\{\fp\in C^\infty(M,\R)\ |\ \gamma^*f=f \mbox{ for any } \gamma\in G \}.
\]
We will denote the underlying infinite dimensional group by $\mc H_{\sigma}^G$ and its orbit by $\mc M^G$.

\begin{ex}[K\"ahler case]\label{ex:Kahler-class}
	If $(M,g,I,J)$ is compact K\"ahler manifold with $I=J$, then $\sigma\equiv 0$, and $\omega:=F_+$ is the K\"ahler form. Hence the group $\mc H$ can be identified with the additive group $C^\infty(M,\R)/\R$, and the action induced by a path $\fp_t\in C^\infty(M,\R)/\R\simeq \mf h$ on a K\"ahler metric is given by
	\[
	\omega\mapsto \omega+dd^c_I\phi,\qquad \phi=-\int_{0}^{t}\fp_s\,ds\in \mc H.
	\]
	Thus, it follows from the $dd^c$-lemma that the orbit $\mc M$ can be identified with the set of K\"ahler manifolds in a given class $[\omega]\in H^{1,1}(M,\R)$ and is isomorphic to an open cone:
	\[
	\mc M=\{\phi \in C^\infty(M,\R)/\R\ |\ \omega+dd^c_I\phi>0 \}\subset \{\alpha\in [\omega]\in H^{1,1}(M,\R) \}.
	\]
\end{ex}

\section{Generalized K\"ahler-Ricci solitons}\label{sec:gks}

Given a generalized K\"ahler manifold $(M,g,I,J)$, we can use the flow construction to deform this structure. It is thus natural to ask whether we can find \emph{distinguished} GK structure in a given deformation class $\mathcal M$. This extends the central problem of K\"ahler geometry, which  is the Calabi question of the existence of  special K\"ahler metrics within a fixed K\"ahler class, see Example~\ref{ex:Kahler-class};  the search for K\"ahler-Einstein metrics, and their close cousins~--- K\"ahler-Ricci solitons, is a concrete example. The following definition is motivated by the study of the generalized Ricci flow (see~\cite{gf-st-20,PCFReg}), and provides a notion of `canonical metric' which naturally extends the theory of Calabi-Yau/K\"ahler-Ricci solitons.

\begin{defn}[Steady gradient generalized K\"ahler soliton]\label{def:gk_soliton}
	A GK manifold $(M,g,I,J)$ is a \emph{steady gradient generalized K\"ahler-Ricci soliton} (GKRS) if there exists a function $f\in C^\infty(M,\R)$ such that
	\begin{equation}\label{eq:soliton_system}
	\begin{cases}
	\Rc -\frac{1}{4}H^2+\nabla^2f=0\\
	d^* H+i_{\nabla f}H=0.
	\end{cases}
	\end{equation}
	The function $f$ is called the \emph{soliton potential} of $(M,g,I,J)$.
\end{defn}

	When  the GK manifold $(M,g,I,J)$ is K\"ahler with $I=\pm J$, we have $H=0$ and the system~\eqref{eq:soliton_system} recovers the usual steady K\"ahler-Ricci soliton equation
	\[
	\Rc+\nabla^2 f=0.
	\]
It is well known that a \emph{compact} steady K\"ahler-Ricci soliton necessarily has $f$ constant. A similar phenomenon occurs for steady GKRS solitons of \emph{symplectic type}:

\begin{prop}\label{pr:symp_soliton}
	Suppose that $(M,g,I,J)$ is a compact GK manifold of symplectic type, which is a GKRS for some $f\in C^\infty(M,\R)$. Then $f \equiv \mathrm{const}$, and $(M,g,I)$, $(M,g,J)$ are K\"ahler Ricci-flat.
\end{prop}
\begin{proof}
	Assume that $I+J$ is globally invertible, and consider $F_+=-2g(I+J)^{-1}$. Let
	\[
	\omega_I=(F_+)_I^{1,1},\quad \beta_I=(F_+)^{2,0}_I
	\]
	be the components of $F_+$ according to the complex $I$-type. We have
	\[
	d^c_I\omega_I=H,
	\]
	thus, since $F_+$ is closed, and $H$ is of type $(2,1)+(1,2)$, we conclude
	\[
	H=d^c_I\omega_I=I\cdot d\omega_I=-I\cdot (\bar{\del}_I \beta_I+\del_I\bar{\beta_I})=d\sqrt{-1}(\bar{\beta_I}-\beta_I).
	\]
	Denote $b:=\sqrt{-1}(\bar{\beta_I}-\beta_I)$, so that $H=db$. Furthermore observe that the second equation of ~\eqref{eq:soliton_system} implies
	\[
	d^*(e^{-f} H)=0.
	\]
	It follows by integration by parts that
	\[
	0=\int_M \la b, d^*(e^{-f}H)\ra_g dV_g=\int_M |db|^2e^{-f}dV_g.
	\]
	Thus $0 = db=H$, so both $(M,g,I)$ and $(M,g,J)$ are steady K\"ahler-Ricci solitons. It remains to note that any compact steady K\"ahler-Ricci soliton is Ricci flat with $f\equiv \mathrm{const}$.
\end{proof}

Proposition~\ref{pr:symp_soliton} implies that if we are seeking nontrivial (i.e., with a non-constant $f$) compact steady GKRS, then we have to allow for both $F_+$ and $F_-$ not being globally defined. The next best thing is having $F_+$ and $F_-$ defined on an open dense set, which is equivalent to $(M,g,I,J)$ being log-nondegenerate.  To begin understanding these structures we recall that on a K\"ahler background the soliton equations imply that the vector field $X=I\nabla f$ is Killing and real-holomorphic. A similar phenomenon occurs on any generalized K\"ahler background:

\begin{prop}[{\cite[Prop.\,4.1]{st-us-20-GH}}]\label{prop:soliton_equiv}
	Let $(M,g,I,J)$ be a GK manifold and $f\in C^\infty(M,\R)$ a smooth function. Then the following are equivalent:
	\begin{enumerate}
		\item $(M,g,I,J)$ is a steady gradient GKRS with the soliton potential $f\in C^\infty(M,\R)$;
		\item The vector field
		\[
		X_I:=\frac{1}{2}I(\theta_I^\sharp-\nabla f)
		\]
		is Killing and real-holomorphic, and
		\[
		(\rho^B_{I})^{1,1}_I=-\mc L_{IX_I}\omega_I
		\]
		\item The vector field
		\[
		X_J:=\frac{1}{2}J(\theta_J^\sharp-\nabla f)
		\]
		is Killing and real-holomorphic, and
		\[
		(\rho^B_{J})^{1,1}_J=-\mc L_{JX_J}\omega_J,
		\]
	\end{enumerate}
	where $\rho^B_I$ (resp.\,$\rho^B_J$) is the Bismut-Ricci curvature of $(M,g,I)$ (resp.\,$(M,g,J)$).
\end{prop}

The vector fields $X_I$ and $X_J$ introduced in Proposition~\ref{prop:soliton_equiv} will play an important role below. The following elementary observation holds on any log-nondegenerate background.

\begin{lemma}[$F_{\pm}$-Hamiltonian potentials]\label{lm:ham_potentials}
	Let $(M^{2n},g,I,J)$ be a GK manifold and consider $f\in C^\infty(M,\R)$. Let
	\[
	X_I:=\frac{1}{2}I(\theta_I^\sharp-\nabla f),\quad X_J:=\frac{1}{2}J(\theta_J^\sharp-\nabla f).
	\]
	Then the vector fields $X_I\pm X_J$ are $F_{\pm}$-Hamiltonian on the locus where the corresponding symplectic forms are defined, with Hamiltonian potentials $\psi_{\pm}\in C^\infty(U_\pm,\R)$ given by
	\begin{equation}\label{eq:psi_pm}
	\psi_{\pm}=\frac{1}{2}\log\det(I\pm J)-f,\quad i_{X_I\pm X_J}F_{\pm}=-d\psi_{\pm}.
	\end{equation}
\end{lemma}
\begin{proof}
	We compute directly using~\eqref{eq:lee_identity}:
	\[
	\begin{split}
	i_{X_J\pm X_I}F_{\pm}=-2g(I\pm J)^{-1}\frac{1}{2}\left(J\theta_I^\sharp\pm J\theta_J^\sharp-(I\pm J)\nabla f\right)=-\frac{1}{2}d\log\det(I\pm J)+df,
	\end{split}
	\]
	as claimed.
\end{proof}

\begin{rmk}
	We can equivalently rewrite the identities~\eqref{eq:psi_pm} defining $\psi_+$ and $\psi_-$ as
	\begin{equation}\label{eq:psi_pm_v2}
	e^{\psi_{\pm}}\frac{F_{\pm}^n}{n!}=e^{-f}dV_g, \quad f\in C^\infty(M,\R).
	\end{equation}
	In particular the volume form on the left hand side which is initially defined only on $U_{\pm}\subset M$, extends to a smooth volume form on $M$. We will prove in Proposition~\ref{prop:gk_soliton_suff} that equations~\eqref{eq:psi_pm_v2} are equivalent to $(M,g,I,J)$ being a gradient steady GKRS.
\end{rmk}

On a gradient steady GKRS, the vector fields $X_I$ and $X_J$ satisfy several important geometric properties.
\begin{prop}\label{pr:XI_XJ_properties}
	Let $(M,g,I,J)$ be a gradient steady GKRS with $X_I$ and $X_J$ as in Proposition~\ref{prop:soliton_equiv}. Then
	\begin{enumerate}
		\item $\mc L_{IX_I}\sigma=\mc L_{JX_J}\sigma=0$,
		\item $\mc L_{X_I}\sigma=\mc L_{X_J}\sigma=0$.
\end{enumerate}
		If $(M,g,I,J)$ is log-symplectic, then
		\begin{enumerate}
		\item[(3)] $\mc L_{X_I}J+\mc L_{X_J}I=0$,
		\item[(4)] $X_I$ and $X_J$ are both $I$ and $J$ holomorphic,
		\item[(5)] $[X_I,X_J]=0$.
\end{enumerate}
		If $(M,g,I,J)$ is compact and log-nondegenerate, then
		\begin{enumerate}
		\item[(6)] $F_+(X_I,X_J)=F_-(X_I,X_J)=0$,
		\item[(7)] $\sigma^{-1}(I X_I,X_J)=\sigma^{-1}(X_I,JX_J)=0$.
	\end{enumerate}
\end{prop}

\begin{proof}
	To prove (1) we use Proposition~\ref{prop:soliton_equiv} and following~\cite[\S 9.3.4]{gf-st-20} find
	\begin{equation}
	\begin{split}
	2\mc L_{IX_I}\sigma=[I,\mc L_{\frac{1}{2}(\theta_J^\sharp-\theta_I^\sharp)}J]+[I,J]g^{-1}(\rho_{I}^B)^{1,1}(\cdot,I\cdot)g^{-1}=0
	\end{split}
	\end{equation}
	where at the first step we used that $IX_I$ and $JX_J$ are respectively $I$ and $J$ holomorphic, in particular $\mc L_{IX_I}J=\mc L_{IX_I-JX_J}J=\frac{1}{2}\mc L_{\theta_J^\sharp-\theta_I^\sharp}J$, and the general identity
	\[
	g((\mc L_{\theta_J^\sharp-\theta_I^\sharp}J)V,W)=\rho_{I}^B(V,[I,J]W)
	\]
	at the last step. Identity $\mc L_{JX_J}\sigma=0$ follows similarly.
	
	To prove the identity $\mc L_{X_I}\sigma=0$ of (2), we consider the Lie derivative of the complex-valued tensor $\sigma_I=\sigma-\sqrt{-1}I\sigma$. Recall that on a GK background $\sigma_I$ is always a holomorphic section of $\Lambda^2(T^{1,0}M)$, so in holomorphic coordinates $\{z_i\}$ one has
	\[
	\sigma_I=\sum_{i,j} f_{ij}\frac{\del}{\del z_i}\wedge \frac{\del}{\del z_j},
	\]
	where $f_{ij}\in \mc O$ are local holomorphic functions. 
	Therefore using part (1)
	\[
	\begin{split}
	0&=\mc L_{IX_I}\sigma_I=\sum_{i,j}f_{ij}\left([IX_I,\frac{\del}{\del z_i}]\wedge \frac{\del}{\del z_j}+\frac{\del}{\del z_i}\wedge[IX_I,\frac{\del}{\del z_j}]\right)\\&=\sum_{i,j}f_{ij}\left(I[X_I,\frac{\del}{\del z_i}]\wedge \frac{\del}{\del z_j}+\frac{\del}{\del z_i}\wedge I[X_I,\frac{\del}{\del z_j}]\right)\\&=\sqrt{-1}\sum_{i,j}f_{ij} \left([X_I,\frac{\del}{\del z_i}]\wedge \frac{\del}{\del z_j}+\frac{\del}{\del z_i}\wedge[X_I,\frac{\del}{\del z_j}]\right)=\sqrt{-1}\mc L_{X_I}\sigma_I,
	\end{split}
	\]
	where going from the first line to the second we used the fact that vector fields $\frac{\del}{\del z_i}$, so $\mc L_{\frac{\del}{\del z_i}}I=0$, and from the second line to the third we used that $X_I$ is real holomorphic, so $[X_I,\frac{\del}{\del z_i}]$ are complex vector fields of type $(1,0)$.
	
	To prove part (3) we observe that from Lemma~\ref{lm:ham_potentials} it follows that in $U_+$ one has
	\[
	\mc L_{X_I+X_J}F_+=0.
	\]
	Since $F_+=-2g(I+J)^{-1}$, and $X_I+X_J$ is Killing, it implies that $\mc L_{X_I+X_J}(I+J)^{-1}=0$, hence $\mc L_{X_I+X_J}(I+J)=0$. Finally, as $X_I$ is $I$-holomorphic and $X_J$ is $J$-holomorphic, in $U_+$
	\[
	\mc L_{X_I}J+\mc L_{X_J}I=0.
	\]
	
	To prove (4), using the fact that $X_J$ is Killing, real $J$-holomorphic, and $0=\mc L_{X_J}\sigma=\frac{1}{2}\mc L_{X_J}([I,J]g^{-1})$ we find
	\[
	0=[\mc L_{X_J}I,J]=[J,\mc L_{X_I}J]=2J\mc L_{X_I}J-\mc L_{X_I}J^2=2J\mc L_{X_I}J.
	\]
	Hence $\mc L_{X_I}J=0$ and similarly $\mc L_{X_J}I=0$ in $U_+$. Since $U_+$ is dense in $M$, the same identity holds globally.
	
	For (5), we observe that since $X_I$ is $I$- and $J$-holomorphic, and Killing, it follows that
	\[
	\mc L_{X_I}\theta_I=\mc L_{X_I}\theta_J=0.
	\]
	Therefore
	\[
	\mc L_{X_I}(-JX_J)=\frac{1}{2}[\theta_J^\sharp-\nabla f,X_I]=[X_I,IX_I]+\frac{1}{2}[\theta_J^\sharp-\theta_I^\sharp,X_I]=0.
	\]
	So $J[X_I,X_J]=[X_I,JX_J]=0$, as required.
	
	Finally, to prove (6) and (7) we use that $X_I+X_J$ is $F_+$-Hamiltonian: $F_+(\cdot,X_I+X_J)=d\psi_+$. Since $X_I$ preserves $F_+$ and commutes with $X_J$ by item (5), we conclude that $d\psi_+(X_I)$ is constant. Now, $M$ is compact and $\psi_+=-f+\frac{1}{2}\log\det(I+J)\colon M\to [-\infty,+\infty)$ attains its maximum, so this constant must be zero, since $d\psi_+$ vanishes at a maximum point. Thus $F_+(X_I,X_J)=F_+(X_I,X_I+X_J)=0$. The same argument applies to $F_-$. 
	
	Next, remembering that $F_+=\sigma^{-1}(J-I)$ and $F_-=\sigma^{-1}(I+J)$, we have
	\[
	\begin{split}
	0&=F_+(X_J,X_I)=\sigma^{-1}((J-I)X_J,X_I),\\
	0&=F_-(X_J,X_I)=\sigma^{-1}((J+I)X_J,X_I).
	\end{split}
	\]
	By adding and subtracting these two identities we get the required vanishing. The lemma is proved.
\end{proof}

From Proposition~\ref{pr:XI_XJ_properties} we know that any gradient steady GKRS has two distinguished vector fields $X_I,X_J$, preserving $I,J$ and $g$, such that vector fields $X_I\pm X_J$ are $F_{\pm}$-Hamiltonian having a prescribed behavior near the \emph{boundary} $Z_\pm=M\backslash U_{\pm}$ expressed through equation~\eqref{eq:psi_pm} (or equivalently~\eqref{eq:psi_pm_v2}). It turns out that conversely, gradient steady GKRS can be characterized by the existence of such vector fields.

\begin{prop}\label{prop:gk_soliton_suff}
	Let $(M^{2n},g,I,J)$ be a log-nondegenerate GK manifold, and assume that there are vector fields $X_I$ and $X_J$ such that:
	\begin{enumerate}
		\item $X_I$ and $X_J$ preserve $I,J$ and $g$ and commute;
		\item $X_I\pm X_J$ are $F_\pm$-Hamiltonian with Hamiltonian potentials $\psi_{\pm}\in C^\infty(U_\pm,\R)$;
		\item The volume forms $e^{\psi_{\pm}}\frac{F^n_{\pm}}{n!}$ on $U_{\pm}$ extend to smooth volume forms on $M$ (equivalently, the functions $\psi_{\pm}-\frac{1}{2}\log\det(I\pm J)$ defined on $U_{\pm}$ extend to smooth functions on $M$);
		\item After normalizing $\psi_\pm$ by adding constants we have
		\begin{equation}\label{eq:volume_forms_equality}
		e^{\psi_+}\frac{F_+^n}{n!}=e^{\psi_-}\frac{F_-^n}{n!}.
		\end{equation}
	\end{enumerate}
	Then $(M,g,I,J)$ is a gradient steady GKRS with soliton potential
	\[
	f=\log\left(\frac{dV_g}{F_+^n/n!}\right)-\psi_+=\log\left(\frac{dV_g}{F_-^n/n!}\right)-\psi_-.
	\]
\end{prop}
\begin{proof}
	We will work on an open dense subset $U_+\cap U_-\subset M$, so that both $F_+, F_-$ as well as $\sigma^{-1}$ are well-defined. Then from the definition of $f$, we have
	\[
	\psi_{\pm}+f=\frac{1}{2}\log\det(I\pm J).
	\]
	It follows from~\eqref{eq:lee_identity} that
	\[
	X_I\pm X_J=-F_{\pm}^{-1}(d\psi_{\pm})=\frac{1}{2}(I\pm J)\left(\frac{1}{2}\nabla \log\det(I\pm J)-\nabla f\right)=\frac{1}{2}I(\theta_I^\sharp-\nabla f)\pm \frac{1}{2}J(\theta_J^\sharp-\nabla f),
	\]
	hence
	\[
	X_I=\frac{1}{2}I(\theta_I^\sharp-\nabla f),\quad X_J=\frac{1}{2}J(\theta_J^\sharp-\nabla f).
	\]
	Thus vector field $X_I=\frac{1}{2}I(\theta_I^\sharp-\nabla f)$ is Killing and $I$-holomorphic, verifying part of the statement (2) of Proposition~\ref{prop:soliton_equiv}. To conclude that $(M,g,I,J)$ is a gradient steady GKRS with the soliton potential $f$, it remains to check that
	\[
	(\rho_{I}^B)^{1,1}=-\mc L_{IX_I}\omega_I.
	\]
	To this end we use the identity \cite[Prop.\,3.9]{ap-st-17}
	\[
	\rho_{I}^B=-\frac{1}{2}dJd\Phi,\quad \Phi:=\log\frac{F_+^n}{F_-^n}.
	\]
	Specifically, using that and the assumption~\eqref{eq:volume_forms_equality}, which implies $\Phi=\psi_--\psi_+$, we compute
	\[
	\begin{split}
	\rho_{I}^B&=\frac{1}{2}dJd(\psi_+-\psi_-)\\
	&=
	\frac{1}{2}dJ(i_{X_I-X_J}F_--i_{X_I+X_J}F_+)\\
	&=
	\frac{1}{2}dJ(i_{X_I}(F_--F_+)-i_{X_J}(F_++F_-))\\&=
	-dJ(i_{X_I}F_+)+\frac{1}{2}dJ i_{X_I-X_J}(F_++F_-)\\
	&=-di_{IX_I}F_+-di_{X_I-X_J}\sigma^{-1}\\
	&=-\mc L_{IX_I}F_+,
	\end{split}
	\]
	where from the first line to the second we used the definition of the Hamiltonian potentials $\psi_{\pm}$, from the fourth line to the fifth we used the identities $F_+(JX,Y)+F_+(X,IY)=0$ and $F_++F_-=2\sigma^{-1}J$, and the final equality holds since $X_I-X_J$ preserves the symplectic form $\sigma^{-1}=2g[I,J]^{-1}$ by Proposition \ref{pr:XI_XJ_properties}.
	
	It remains to observe that $(F_+)^{1,1}_{I}=\omega_I$, so that 
	\[
	\mc L_{IX_I} \omega_I=(\mc L_{IX_I}F_+)^{1,1}_I=-(\rho^B_I)^{1,1}_I 
	\]
	since $IX_I$ is $I$-holomorphic.
\end{proof}
\begin{rmk} The soliton potential $f$ is defined up to an additive constant. We can eliminate this freedom by an additional normalization
\[ \int_M e^{\psi_+} \frac{F_+^{n}}{n!} = \int_M e^{\psi_-} \frac{F_-^{n}}{n!} = \int_M e^{-f}dV_g = {\rm Vol}(M, g_0), \]
where $g_0$ is some background Riemannian metric.
\end{rmk}

\section{The weighted \texorpdfstring{$\pmb{J}$}--functional}\label{sec:aubin}
\subsection{The weighted \texorpdfstring{$\pmb{J}$}--functional in K\"ahler geometry}

Here we recall the definition of the well-known (weighted) $\pmb{J}$-functional,  in the case when $m_0:=(M, J, g_0, \omega_0)$ is a fixed K\"ahler manifold, and $X$ is a given Killing vector field with zeroes. In this case, we let $G$ be the compact torus in the isometry group of $g_0$ generated by the flow of $X$, and  $\mc M_{m_0}^G$ the corresponding orbit of $G$-invariant K\"ahler metrics in the K\"ahler class of $\omega_0$, see Example~\ref{ex:Kahler-class}. We think of $\mc M_{m_0}^G$ as a formal Fr\'echet manifold, with tangent space at $(g, \omega)\in \mc M_{m_0}^G$ identified with $C^{\infty}(M, \R)/\R$. Aubin
~\cite{Aub} introduced a functional  $\pmb{J}_{m_0}: {\mc M}_{m_0} \to \R, \, \pmb{J}(m_0)=0$, through its derivative 
\[ \pmb{\mu} (\fp) := \int_M \fp \left(\frac{\omega_0^n}{n!} - \frac{\omega^{n}}{n!}\right), \]
seen as a closed $1$-form on the space $\mc M_{m_0}$.  This space is contractible, thus $\pmb{\mu}$ is exact, yielding the functional $\pmb{J}$.  A  suitable $X$-weighted version of $\pmb{J}_{m_0}$ is introduced in \cite{TZ1, HL} in order to account for the theory of KRS, as follows: for any $(g, \omega) \in \mc M_{m_0}$, we let $\psi_{\omega}$ denote the unique Killing potential of $X$ (i.e. smooth function such that $X=J {\rm grad}_g \psi_{\omega}$) satisfying  $\int_M \psi_{\omega} \frac{\omega^n}{n!} =c$ for some fixed constant $c$, and let $\pmb{J}_{X, m_0}$ be the primitive of the closed $1$-form
\[  \pmb{\mu}_{X} (\fp) := \int_M \fp \left(e^{\psi_{\omega_0}}\frac{\omega_0^n}{n!} - e^{\psi_{\omega}}\frac{\omega^{n}}{n!}\right), \]
on the space $\mc M_{m_0}^G$. Notice that the normalization of $\psi_{\omega}$  is equivalent to require that  $\psi_{\omega}(M)$ is a fixed interval (depending on $c$). It then follows by the Duistermaat-Heckman theorem that the normalization constant $c$ can be chosen so that  $\int_M e^{\psi_{\omega}}\frac{\omega^n}{n!}=\int_M e^{\psi_{\omega_0}} \frac{\omega_0^n}{n!} = {\rm Vol}(M, \omega_0)$, see \cite{TZ1}.

\subsection{The weighted \texorpdfstring{$\pmb{J}$}--functional in log-nondegenerate GK geometry}

Consider a compact log-nondegenerate GK manifold $m_0=(M,g_0,I_0,J)$. The goal of this section is to define a functional on the universal cover $\mc M^G$ of the orbit of the flow construction $\mc M_{m_0}^G$,  reminiscent to  the (weighted) $\pmb{J}$-functional in the K\"ahler setting introduced above. It will depend on a choice of two distinguished commuting, quasiperiodic vector fields which span a torus $G$, and representing the prospective soliton vector fields $X_I$ and $X_J$ of Proposition~\ref{prop:soliton_equiv}. Specifically, for a log-nondegenerate GK manifold $(M,g,I,J)$ we fix two vector fields $X_I,X_J\in \Gamma(TM)$ such that:

\begin{enumerate}
	\item $X_I$ and $X_J$ preserve $I,J$ and $g$ and commute.
	\item $X_I\pm X_J$ are $F_{\pm}$-Hamiltonian with Hamiltonian potentials $\psi_{\pm}\in C^\infty(U_{\pm},\R)$:
	\[
	i_{X_I\pm X_J}F_{\pm}=-d\psi_{\pm}.
	\]
	\item The volume forms $e^{\psi_{\pm}}\frac{F^n_{\pm}}{n!}$ extend to smooth volume forms on $M$.
	\item $F_{\pm}$-Hamiltonian potentials $\psi_{\pm}$ are normalized in such a way that
	\[
	\int_M e^{\psi_{\pm}}\frac{F_{\pm}^{n}}{n!}={\rm Vol}(M, g_0).
	\]
\end{enumerate}
Let $G\subset \mathrm{Aut}(M,g,I,J)$ be the Lie subgroup generated by the vector fields $X_I$ and $X_J$. Since the vector fields $X_I$ and $X_J$ commute, and $G$ is a subgroup of the isometry group, which is compact, $G\subset \mathrm{Aut}(M,g,I,J)$ is a torus. Denote by $\mc H_\sigma^G$ the infinite-dimensional Lie group with the Lie algebra $C^\infty(M,\R)^G/\R$ underlying the $G$-equivariant flow construction. 

We next introduce a functional on the universal cover of $\mc M^G$
\[
\pmb J_{X_I,X_J,m_0}\colon \til{\mc M^G_{m_0}}\to \R
\]
through its variation along the flow construction, such that its critical points are given by gradient steady GKRS.  In the case when $(M,g,I,J)$ is nondegenerate and $X_I=X_J=0$, this functional  will reduce to the one that was introduced by the first two authors in~\cite[\S 4]{ap-st-17}. Our construction builds upon this paper, taking careful consideration of the degeneracy loci of $\gs$ (note that our definition of $\sigma$ differs by a factor of 2 from~\cite{ap-st-17}).

Before we introduce the functional, let us make an observation about the $G$-equivariant flow construction. 
\begin{lemma} \label{l:flowinvarance}
	For any $\fp\in C^\infty(M,\R)^G$ and the corresponding $\gs$-Hamiltonian vector field $X_{\fp}$, we have
	\[
	[X_{\fp},X_I]=[X_{\fp},X_J]=0.
	\]
	In particular, along the Hamiltonian flow construction generated by a family $\fp_t\in C^\infty(M,\R)^G$ vector fields $X_I$ and $X_J$ remain holomorphic and Killing.
\end{lemma}
\begin{proof}
We have $[X_{\fp},X_I]=\mc L_{X_I}(\sigma(d\fp))=\sigma(d\la \fp, X_I\ra)=0$, where in the second identity we used the fact that $X_I$ preserves $\sigma$, and the final term vanishes, since the potential $\fp$ is invariant under the action of the group $G$, which is generated by the infinitesimal flows of $X_I$ and $X_J$.

Thus along the flow construction generated by $\fp_t\in C^\infty(M,\R)^G$, the vector fields $X_I$ and $X_J$ remain jointly $I$ and $J$ holomorphic, since $J$ does not change, and $I$ flows by a diffeomorphism generated by $X_{\fp_t}$. Finally, since $\dt g=-(dd^c_I\fp_t J)^{\mathrm{sym}}$, and the term on the right-hand side is annihilated by $\mc L_{X_I}$ and $\mc L_{X_J}$, the vector fields $X_I$ and $X_J$ also remain Killing.
\end{proof}

Next we compute the evolution of the $F_{\pm}$-Hamiltonian potentials $\psi_{\pm}$ along the flow.

\begin{lemma}\label{lm:flow_volume_variation}
	Along the flow construction generated by $\fp\in C^\infty(M,\R)^G$ we have
	\[
	\left. \dt \right|_{t=0}{\psi_{\pm}}=\mp\langle Id\fp, (X_I\pm X_J)\rangle,
	\]
	\[
	\left. \dt \right|_{t=0}\left(
		e^{\psi_{\pm}}\frac{F_{\pm}^{n}}{n!}
		\right)
		=\mp d\left(e^{\psi_\pm} Id\fp\wedge \frac{F_\pm^{n-1}}{(n-1)!}\right).
	\]
\end{lemma}
\begin{proof}
	We provide the argument for $\psi_+$, with the case of $\psi_-$ being analogous. For all times along the flow construction, the function $\psi_+$ satisfies
	\[
	d\psi_+=-i_{X_I+X_J}F_+.
	\]
	Taking first order variation, we have
	\[
	\left. \dt \right|_{t=0}d{\psi_+}=i_{X_I+X_J}dId\fp=-d(i_{X_I+X_J}Id\fp),
	\]
	where in the second identity we used that $\mc L_{X_I+X_J}(Id\fp)=0$ by Lemma \ref{l:flowinvarance}.
	
	It remains to check that with the variation of the Hamiltonian potential given by $\dt \big|_{t=0}\psi_+=-\langle Id\fp, (X_I+ X_J)\rangle$, the variation of the total volume $\int_M e^{\psi_+} \frac{F_+^{n}}{n!}$ vanishes. Indeed, it is easy to see that the variation of the volume form is exact:
	\[
	\begin{split}
	\left. \dt \right|_{t=0}{\left(e^{\psi_+} \frac{F_+^{n}}{n!}\right)}&=-\langle Id{\fp}, (X_I+X_J)\rangle e^{\psi_+} \frac{F_+^{n}}{n!}-e^{\psi_+} dId\fp\wedge  \frac{F_+^{n-1}}{(n-1)!}\\
	&=-e^{\psi_+}(d\psi_+\wedge Id\fp+dId{\fp})\wedge \frac{F_+^{n-1}}{(n-1)!}=-d\left(e^{\psi_+}Id{\fp}\wedge \frac{F_+^{n-1}}{(n-1)!}\right),
	\end{split}
	\]
	where we used that for any 1-form $\alpha$, \[
	n\cfrac{d\psi_+\wedge \alpha\wedge F_+^{n-1}}{F_+^n}
	=\langle \alpha, (X_I+X_J)\rangle.\]
\end{proof}

\begin{defn}
	Let $m=(M,g,I,J)$ be a compact log-nondegenerate GK manifold, equipped with two vector fields $X_I$ and $X_J$ satisfying the above assumptions. Define $\pmb\nu_{X_I, X_J}$ to be a 1-form on $\mc M^G$ given at a point $m$ and $\fp\in C^\infty(M,\R)^G/\R\simeq T_m\mc M^G$ by
	\[
	\pmb\nu_{X_I, X_J}(\fp):=\int_M \fp \left(
		e^{\psi_+}\frac{F_+^n}{n!}-e^{\psi_-}\frac{F_-^n}{n!}
	\right), 
	\]
	where $\psi_{\pm}$ are the $F_{\pm}$-Hamiltonians of $X_I \pm X_J$, normalized by
	\[ \int_M e^{\psi_+}\frac{F_+^n}{n!}= \int_M e^{\psi_-}\frac{F_-^n}{n!}={\rm Vol}(M, g_0).\]
\end{defn}

We are going to prove that the 1-form $\pmb\nu:=\pmb\nu_{X_I, X_J}$ is closed, so that we can integrate it along paths in $\mc M^G$ and obtain a functional $\pmb{J}:=\pmb J_{X_I, X_J}$ on the universal cover $\pi\colon \til{\mc M^G}\to \mc M^G$, such that $d\pmb J_{X_I, X_J}=\pi^*\pmb\nu_{X_I, X_J}$.

\begin{prop}\label{prop:closed_1_form}
	The 1-form $\pmb\nu_{X_I, X_J}$ on $\mc M^G$ is closed.
\end{prop}
\begin{proof}
	The argument is an extension of~\cite[Lemma\,4.4]{ap-st-17} in the case when $X_I=X_J=0$. Let $\fp_1,\fp_2\in C^\infty(M,\R)^G$ be two flow potentials inducing the $\sigma$-Hamiltonian vector fields $X_{\fp_1}$, $X_{\fp_2}$ on $M$. By pushing them forward via the action of $\mc H_{\sigma}^G\subset \mathrm{Diff}(M)$, we obtain 
	right-invariant vector fields on $\mc H_\sigma^G$.
	By abuse of notation we denote the corresponding fundamental vector fields on $\mc M^G$ by $\fp_1$ and $\fp_2$.
	
	To prove that $\pmb\nu=\pmb\nu_{X_I, X_J}$ is closed we want to show that
	\[
	d\pmb\nu(\fp_1,\fp_2)=\fp_1\cdot \pmb\nu(\fp_2)-\fp_2\cdot \pmb\nu(\fp_1)-\pmb\nu([\fp_1,\fp_2])=0.
	\]
	Denote by $\dt$ the infinitesimal variation under the flow construction induced by $\fp_1$. Then with the use of Lemma~\ref{lm:flow_volume_variation} we directly compute:
	\[
	\begin{split}
	\fp_1 \cdot \pmb\nu(\fp_2)=&\
	\int_M \fp_2 \dt \left(
	e^{\psi_+}\frac{F_+^n}{n!}-e^{\psi_-}\frac{F_-^n}{n!}
	\right)\\
	=&\
	-\int_M \fp_2 d\left(e^{\psi_+}Id{\fp_1}\wedge \frac{F_+^{n-1}}{(n-1)!}
	+e^{\psi_-}Id\fp_1\wedge\frac{F_-^{n-1}}{(n-1)!}
	\right)\\
	=&\ \int_M d\fp_2\wedge Id\fp_1\wedge \left(
		e^{\psi_+}\frac{F_+^{n-1}}{(n-1)!}+
		e^{\psi_-}\frac{F_-^{n-1}}{(n-1)!}
		\right)\\
	=&\ \frac{1}{2}\int_M \la (I+J)d\fp_2, I d\fp_1\ra e^{\psi_+}\frac{F_+^n}{n!}+\frac{1}{2}\int_M \la (I-J)d\fp_2, I d\fp_1\ra e^{\psi_-}\frac{F_-^n}{n!}\\
	=&\ 
	\frac{1}{4}\int_M\la(I+J)d\fp_1,(I+J)d\fp_2\ra e^{\psi_+}\frac{F_+^n}{n!}+
	\frac{1}{4}\int_M\la(I-J)d\fp_1,(I-J)d\fp_2\ra e^{\psi_-}\frac{F_-^n}{n!}\\
	&+\frac{1}{2}\int_M \{\fp_1,\fp_2\}_{\sigma}\left(
	e^{\psi_+}\frac{F_+^n}{n!}-e^{\psi_-}\frac{F_-^n}{n!}
	\right).
	\end{split}
	\]
	
	Where, while integrating by parts, we used that $e^{\psi_{\pm}}F_\pm^n$ extends to smooth volume form on $M$, which implies that $e^{\psi_{\pm}}F_{\pm}^{n-1}=\frac{1}{n}i_{F_{\pm}^{-1}}\left(e^{\psi_{\pm}}F_{\pm}^{n}\right)$ is a smooth $(n-2)$ form globally defined on $M$. The first two terms of the final expression are symmetric with respect to $\fp_1$ and $\fp_2$ and will cancel out with the corresponding terms in $\fp_2\cdot\pmb\nu(\fp_1)$. To prove that  the last term will cancel out with $\pmb\nu([\fp_1,\fp_2])$, we need to check that the vector field $[\fp_1,\fp_2]$ on $\mc M^G$ is induced by the function $\{\fp_1,\fp_2\}_\sigma$. Indeed, since $\sigma$ is generically invertible, $\mc H_{\sigma}$ can be thought of as a subgroup of $\mathrm{Diff}(M)$ acting on $M$ by pull-backs. Then the commutator of right-invariant vector fields $\fp_1$ and $\fp_2$ on $\mc H_{\sigma}$ corresponds to the \emph{anti-commutator} of the corresponding $\sigma$-Hamiltonian vector fields $X_{\fp_1}$ and $X_{\fp_2}$. It remains to recall that
	\[
	[X_{\fp_1},X_{\fp_2}]=-X_{\{\fp_1,\fp_2\}_\sigma}.
	\]
	Hence $[\fp_1,\fp_2]=\{\fp_1,\fp_2\}_\sigma$.
\end{proof}

\begin{prop}[$\pmb{J}$-functional] \label{p:Jfunctional}
	Suppose $m_0:=(M,g_0,I_0,J)$ is a log-nondegenerate generalized K\"ahler structure equipped with vector fields  $X_I,X_J$ satisfying
		\begin{enumerate}
		\item $X_I$ and $X_J$ preserve $I,J$ and $g$.
		\item $X_I\pm X_J$ admit $F_{\pm}$-Hamiltonian potentials $\psi_{\pm}\in C^\infty(U_{\pm},\R)$:
		\[
		i_{X_I\pm X_J}F_{\pm}=-d\psi_{\pm}.
		\]
		\item The volume forms $e^{\psi_{\pm}}\frac{F^n_{\pm}}{n!}$ extend to smooth volume forms on $M$.
		\item $\psi_{\pm}$ are normalized so that
		\[
		\int_M e^{\psi_{\pm}}\frac{F^n_{\pm}}{n!}=\int_M dV_g.
		\]
	\end{enumerate}	

	Denote by $G\subset \mathrm{Aut}^0(M,g,I,J)$ the compact torus generated by $X_I,X_J$,  and let $\mc M^G$ be the $G$-invariant generalized K\"ahler class of $m_0$. Then there exists a unique functional
	\[
	\pmb J_{X_I,X_J,m_0}\colon \til{\mc M^G}\to \R,
	\]
	on the universal cover $\til{\mc M^G}\to \mc M^G$ such that
	\begin{itemize}
		\item $\pmb J_{X_I,X_J,m_0}(m_0)=0$;
		\item For any $m\in\til{\mc M^G}$ and $\fp\in T_m \til{\mc M^G}\simeq C^\infty(M,\R)^G/\R$,
		\[
		d \pmb J_{X_I,X_J,m_0}[\fp]=\pmb \nu_{X_I, X_J}(\fp).
		\]
		\item A GK structure $m \in \til{\mc M^G}$ is a gradient steady GKRS with associated vector fields $X_I, X_J$ if and only if it is critical point of $\pmb J_{X_I,X_J,m_0}$.
	\end{itemize}
\end{prop}

\begin{proof}
	By Proposition~\ref{prop:closed_1_form}, 1-form $\pmb \nu_{X_I, X_J}$ on $\mc M^G$ is closed. Thus we can integrate it along paths originating at $m_0$ and get a genuine function on the universal cover $\til{\mc M^G}$, which by construction satisfies the first two claimed properties.  By the form of $\pmb \nu_{X_I,X_J}$ and Proposition~\ref{prop:gk_soliton_suff} the third property holds.
\end{proof}

\begin{rmk}
If we choose a different base point $m_0'$, then the difference $\pmb J_{X_I,X_J,m_0}-\pmb J_{X_I,X_J,m_0'}$ will be a constant. Since we are mostly interested in the variations of the $J$-functional, we often will suppress the explicit choice of the base point in the notation.
\end{rmk}

The implicit and potentially complicated topology of the space $\mc M^G$ makes it challenging to study the functional $\pmb J_{X_I,X_J}$. To overcome this issue, we can consider a distinguished contractible subset \[
\mc M^G_{c}\subset \mc M^G
\]
obtained via the flow construction (Construction~\ref{cstr:flow}) applied to the base point structure $(M,g_0,I_0,J_0)$ along the constant speed flow potentials $\fp_t$, i.e., $\fp_t=\fp_0$. Specifically, for every fixed $\fp\in C^\infty(M,\R)^G$ we consider a one parameter family of GK structures $(g_t,I_t,J)$ with
\[
\dot {(F_+)}_t = -dd^c_{I_t}\fp\quad \dot I_t=X_{\fp}.
\]
A key point is that the functional $\pmb J_{X_I,X_J}$ is convex along such paths.

\begin{lemma}\label{lm:2nd_variation}
Let $\fp_t\in C^\infty(M,\R)$ be a smooth one-parameter family of flow potentials and denote by $(g_t,I_t,J)$ the flow construction induced by this family. Then we have
\begin{equation}\label{eq:2nd_variation}
\begin{split}
\frac{d^2}{dt^2}\pmb J_{X_I,X_J}(g_t,I_t,J)=&\
\int_M \frac{d\fp_t}{dt} \left(
e^{\psi_+}\frac{F_+^n}{n!}-e^{\psi_-}\frac{F_-^n}{n!} 
\right)\\
&\ +
	\frac{1}{4}\int_{M}|(I+J)d\fp_t|^2 e^{\psi_+}\frac{F_+^n}{n!}
	+
	\frac{1}{4}\int_{M}|(I-J)d\fp_t|^2 e^{\psi_-}\frac{F_-^n}{n!}.
\end{split}
\end{equation}
In particular, if $(g_t, I_t, J)$ denotes the flow construction along a constant speed path, then $\pmb J_{X_I, X_J} (g_t, I_t, J)$ is convex.
\begin{proof} For a one-parameter family $(g_t, I_t, J)$ given by the flow construction, we compute
	\begin{equation}
	\begin{split}
	\frac{d^2}{dt^2}\pmb J_{X_I,X_J}(g_t,I_t,J)=&\ 
	\frac{d}{dt}\int_M \fp_t \left(
	e^{\psi_+}\frac{F_+^n}{n!}-e^{\psi_-}\frac{F_-^n}{n!}\right)
	\\
	=&\ \int_M \frac{\del \fp_t}{\del t} \left(
	e^{\psi_+}\frac{F_+^n}{n!}-e^{\psi_-}\frac{F_-^n}{n!}\right)
	+
	\int_M \fp_t \frac{\del}{\del t}\left(
	e^{\psi_+}\frac{F_+^n}{n!}-e^{\psi_-}\frac{F_-^n}{n!}\right)\\
	=&\ 
	\int_M \frac{\del \fp_t}{\del t} \left(
	e^{\psi_+}\frac{F_+^n}{n!}-e^{\psi_-}\frac{F_-^n}{n!} 
	\right)\\
	&\ +
	\frac{1}{4}\int_{M}|(I+J)d\fp_t|^2 e^{\psi_+}\frac{F_+^n}{n!}
	+
	\frac{1}{4}\int_{M}|(I-J)d\fp_t|^2 e^{\psi_-}\frac{F_-^n}{n!}.
	\end{split}
	\end{equation}
	
	Where in the last equality we used the computation of Proposition~\ref{prop:closed_1_form}.
\end{proof}
\end{lemma}

The properties of the functional $\pmb J_{X_I,X_J}$ immediately imply the rigidity of GKRS along the flow construction.  This is Theorem \ref{thm:rigidite_soltions} of the introduction, which we restate for convenience:
\begin{thm}[Rigidity of solitons]\label{thm:rigid_soltions}
	Let $(M,g_t,I_t,J), t \in [0, \tau)$ be a one-parameter family of compact log-nondegenerate GKRS lying in $\mc M^G$, where $G$ is generated by the vector fields $X_I, X_J$ associated to $(g_0, I_0, J)$.  Then $(M,g_t,I_t,J)=(M,g_0,I_0,J)$ for all $t$.
\end{thm}

\begin{proof}
	Assume that the family $(M,g_t, I_t,J)$ is induced by the family of flow potentials $\fp_t\in C^\infty(M,\R)$. Then by Proposition~\ref{p:Jfunctional} the first order variation of $\pmb J_{X_I,X_J}$ vanishes at a gradient steady GKRS:
	\[
	\frac{d}{dt}\pmb J_{X_I,X_J}(g_t,I_t,J)=0.
	\]
	Therefore, using the formula~\eqref{eq:2nd_variation} for the second variation along the flow $\fp_t$
	we find:
	\[
	\begin{split}
	0&=\frac{d^2}{dt^2}\pmb J_{X_I,X_J}(g_t,I_t,J)=
	\int_M \frac{\del \fp_t}{\del t} \left(
	e^{\psi_+}\frac{F_+^n}{n!}-e^{\psi_-}\frac{F_-^n}{n!} 
	\right)\\
	&+
	\frac{1}{4}\int_{M}|(I+J)d\fp_t|^2 e^{\psi_+}\frac{F_+^n}{n!}
	+
	\frac{1}{4}\int_{M}|(I-J)d\fp_t|^2 e^{\psi_-}\frac{F_-^n}{n!}.
	\end{split}
	\]
	The first term on the right hand side vanishes, since on a soliton $e^{\psi_+}\frac{F_+^n}{n!}=e^{\psi_-}\frac{F_-^n}{n!}$, and the remaining terms are positive unless $d \fp_t \equiv 0$.  Therefore the flow construction is trivial and $(M,g_t,I_t,J)=(M,g_0,I_0,J)$.
\end{proof}

\section{Existence and Uniqueness of GKRS on Hopf surfaces}\label{sec:hopf}

In this section we analyze which Hopf surfaces admit GK structures, and prove uniqueness of GKRS on those Hopf surfaces which admit a GK structure.

\subsection{Existence of GK structures on Hopf surfaces}

As was mentioned in Section~\ref{sec:bg}, by~\cite{ap-gu-07}, a Hopf surface $(M,I)$ admits an \emph{odd} GK structure if and only if $(M,I)$ is a quotient of a diagonal Hopf surface $(\til M, I_{\alpha\beta})$ by a (possibly trivial) cyclic group $\Gamma_{k,l}\subset \mathrm{Aut}(\til M,I_{\alpha\beta})$ {given by \eqref{Gamma}}. Next we determine which Hopf surfaces $(M,I)$ admit an \emph{even} GK structure. It turns out that we have a result analogous to the above:

\begin{thm}\label{thm:hopf_existence_even} 
	If a Hopf surface $(M,I)$ is a part of an even GK structure $(M,g,I,J)$, then $(M,I)$ is biholomorphic a quotient of the diagonal Hopf surface $(\til M,I_{\alpha\beta})$ by a cyclic group $\Gamma_{k,l}\simeq \Z_\ell$. Conversely, any such Hopf surface is a part of an even GK structure.
\end{thm}
Before we turn to the proof this theorem we will state and prove several results regarding the geometry of Hopf surfaces. First we recall the description of the automorphism group of diagonal Hopf surfaces $(M,I_{\alpha\beta})$.
\begin{prop}[{\cite[\S 2]{na-74}}]\label{prop:hopf_aut}
	Let $(M,I_{\alpha\beta})$ be a diagonal Hopf surface with parameters $(\alpha,\beta)$, and denote by $A\in \mathrm{GL}_2(\C)$ the contraction~\eqref{eq:hopf_def_A}. Then
	\begin{enumerate}
		\item if $\alpha=\beta$, then $\mathrm{Aut}(M,I_{\alpha\beta})=\mathrm{GL}_2(\C)/\la A\ra$,
		
		\item if $\alpha=\beta^q$ for some integer $q\geq 2$, then $\mathrm{Aut}(M,I_{\alpha\beta})=G/\la A\ra$, where the group $G\simeq (\C^*)^2\rtimes\C$ acts on $\C^2\backslash\{\mathbf{0}\}$ as
		\[
		(z_1,z_2)\mapsto (az_1+bz_2^q,dz_2),\quad a,d\in\C,^*\ b\in\C,
		\]
		
		\item if $\alpha\neq\beta^q$ and $\beta\neq\alpha^q$ for any integer $q\geq 1$, then $\mathrm{Aut}(M,I_{\alpha\beta})=(\C^*)^2/\la A\ra$, where $(\C^*)^2$ acts on $\C^2\backslash\{\mathbf{0}\}$ via coordinate-wise multiplication.
	\end{enumerate}
\end{prop}

Next we study the action of $\mathrm{Aut}(M,I_{\alpha\beta})$ on auxiliary geometric objects underlying odd and even GK structures~--- holomorphic splittings $T^{1,0}M=L_1\oplus L_2$, and disconnected canonical divisors $E_1\cup E_2$. On every diagonal Hopf surface $(M,I_{\alpha\beta})=(\C^2\backslash\{\mathbf 0\})/\la A\ra$ there is a \emph{standard} splitting with $L^{\mathrm{std}}_i=\mathrm{span}(\frac{\del}{\del z_i})$, and a \emph{standard} disconnected canonical divisor $E^{\mathrm{std}}_1\cup E^{\mathrm{std}}_2$ with $E^{\mathrm{std}}_i=\{z_i=0\}/\la A\ra$. The following proposition proves that up to the order, and the action of $\mathrm{Aut}(M,I_{\alpha\beta})$ a splitting of $T^{1,0}M$ / disconnected canonical divisor $E_1\cup E_2$ are equivalent to the standard ones.

\begin{prop}\label{prop:hopf_Q_sigma}
	Let $(M,I_{\alpha\beta})$ be a diagonal Hopf surface with parameters $(\alpha,\beta)$. Then $\mathrm{Aut}(M,I_{\alpha\beta})$ acts transitively on the set of
	\begin{enumerate}
		\item holomorphic splittings $T^{1,0}M=L_1\oplus L_2$ (ignoring the order of the summands);
		
		\item disconnected canonical divisors $E_1\cup E_2$ (ignoring the order of the components).
	\end{enumerate}
	Furthermore,
	\begin{enumerate}
		\item[(i)] for the standard holomorphic splitting $T^{1,0}M=L^{\mathrm{std}}_1\oplus L^{\mathrm{std}}_2$ with the prescribed order of summands, the subgroup of $\mathrm{Aut}(M,I_{\alpha\beta})$ preserving the splitting is $(\C^*)^2/\la A\ra$;
		
		\item[(ii)] for the standard disconnected canonical divisor $E^{\mathrm{std}}_1\cup E^{\mathrm{std}}_2$ with the prescribed order of summands, the subgroup of $\mathrm{Aut}(M,I_{\alpha\beta})$ preserving the divisor is $(\C^*)^2/\la A\ra$.
	\end{enumerate}
\end{prop}
\begin{proof}
	A holomorphic splitting $T^{1,0}M=L_1\oplus L_2$ can be described by a holomorphic section of ${\rm End}(T^{1,0}M)$
	\[
	Q\colon T^{1,0}M\to T^{1,0}M,\quad Q^2=\mathrm{Id}.
	\]
	with nontrivial $\pm 1$ eigenspaces. We lift $Q$ to an operator on the tangent bundle of $\C^2\backslash \{\mathbf{0}\}$. Since $Q$ is holomorphic, by Hartog's theorem, $Q$ extends through $\mathbf{0}\in\C^2$,  and is thus representable by an absolutely convergent power series around $\mathbf{0}$:
	\[
	Q=\sum_{m,n\geq 0} Q_{mn}z_1^mz_2^n,\quad Q_{mn}\colon \C^2\to\C^2.
	\]
	As $Q$ must be invariant under the contraction $(z_1,z_2)\mapsto (\alpha z_1,\beta z_2)$, this leads to the following cases:
	\begin{enumerate}
		\item If $\alpha=\beta$, all coefficients $Q_{mn}$ except the constant $Q_{00}$ must be zero. Thus $Q$ is just a linear operator $\C^2\to\C^2$ with eigenvalues $\{1,-1\}$, and $\mathrm{GL}_2(\C)$ acts transitively on such operators, so that any operator $Q$ is equivalent to the standard one
		\[
		Q_{\mathrm{std}}:=dz_1\otimes \frac{\del}{\del z_1}-dz_2\otimes \frac{\del}{\del z_2}.
		\]
		
		\item If $\alpha=\beta^q$, $q\geq 2$, then up to a sign $Q$ must be of the form
		\begin{equation*}
		Q=dz_1\otimes \frac{\del}{\del z_1}-dz_2\otimes \frac{\del}{\del z_2}+b z_2^{q-1}dz_1\otimes \frac{\del}{\del z_2},\quad b\in\C.
		\end{equation*}
		It is clear from case (2) of Proposition~\ref{prop:hopf_aut} that $\mathrm{Aut}(M,I_{\alpha\beta})$ acts transitively on the operators of the above form, so any such operator is equivalent  under the action of  $\mathrm{Aut}(M,I_{\alpha\beta})$ to $Q_{\mathrm{std}}$.
		
		\item If $\alpha\neq\beta^q$ and $\beta\neq\alpha^q$ for any integer $q\geq 1$, then all coefficients $Q_{mn}$ except for $Q_{00}$ vanish; since the linear operator $Q_{00}$ must commute with $A=\left(\begin{matrix}
		\alpha & 0 \\ 0 & \beta
		\end{matrix}\right)$, up to a sign,  $Q$ then equals to the standard operator $Q_{\mathrm{std}}$.
	\end{enumerate}
	
	Next we address the action of $\mathrm{Aut}(M,I_{\alpha\beta})$ on disconnected canonical divisors. Any such divisor is given as the zero locus of a holomorphic bivector $\sigma_I\in \Gamma(\Lambda^{2}(T^{1,0}M))$. Lifting $\sigma_I$ to $\C^2\backslash\{\mathbf 0\}$, extending it over $\mathbf{0}$ with the use of Hartog's theorem, and expressing through a convergent power series as above, we write
	\[
	\sigma_I=\sum_{m,n\geq 0} (\sigma_I)_{mn}z_1^mz_2^n,\quad (\sigma_I)_{mn}\in \Lambda^2(\C^2).
	\]
	Invoking again the invariance of $\sigma_I$ under the contraction $(z_1,z_2)\mapsto (\alpha z_1,\beta z_2)$ we conclude
	\begin{enumerate}
		\item If $\alpha=\beta$, all coefficients $(\sigma_I)_{mn}$ with $m+n\neq 2$ must vanish. Thus $\sigma_I$ is a bivector of the form
		\[
		\sigma_I=q(z_1,z_2)\frac{\del}{\del z_1}\wedge \frac{\del}{\del z_2},
		\]
		where $q(z_1,z_2)$ is a quadratic form. Since the zero locus of $\sigma_I$ on $M$ is disconnected, $q(z_1,z_2)$ must be nondegenerate, and up to the action of $\mathrm{Aut}(M,I_{\alpha\beta})=GL_2(\C)/\la A\ra$, $\sigma_I$ is given by a complex multiple of the standard Poisson tensor
		\[
		\pmb\pi=z_1z_2 \frac{\del}{\del z_1}\wedge \frac{\del}{\del z_2}.
		\]
		
		\item If $\alpha=\beta^q$, $q\geq 2$, then $\sigma_I$ must of of the form
		\[
		\sigma_I = c\left(z_1z_2+bz_2^{q+1} \right) \frac{\del}{\del z_1}\wedge \frac{\del}{\del z_2},\quad b,c\in\C, c\neq 0.
		\]
		Using again part of (2) of Proposition~\ref{prop:hopf_aut}, we conclude that up to the action of $\mathrm{Aut}(M,I_{\alpha\beta})$, the bivector $\sigma_I$ is equivalent to a complex multiple of $\pmb\pi$.
		
		\item If $\alpha\neq\beta^q$ and $\beta\neq\alpha^q$, then $\sigma_I$ must be immediately equal to a complex multiple of the $\pmb\pi$.
	\end{enumerate} 
	In either of the above cases, the zero locus of $\sigma_I$ up to the action of $\mathrm{Aut}(M,I_{\alpha\beta})$ is the standard disconnected canonical divisor $(\{z_1=0\}\cup \{z_2=0\})/\la A\ra$.
	
	Claims (i) and (ii) follow directly from the explicit description of $\mathrm{Aut}(M,I_{\alpha\beta})$ in Proposition~\ref{prop:hopf_aut}. This concludes the proof of Proposition~\ref{prop:hopf_Q_sigma}.
\end{proof}

\begin{proof}[Proof of Theorem~\ref{thm:hopf_existence_even}]
	We first assume that $(M,I)$ is a (possibly secondary) Hopf surface, and let $(\til M,I)$ be its primary finite cover. Let $(M,g,I,J)$ be an even GK structure on a Hopf surface $(M,I)$, and $(\til M,g,I,J)$ the corresponding GK structure on the cover $\til M\to M$. Clearly $I$ is different from $\pm J$, since $(\til M,g,I)$ is not K\"ahler. Furthermore, since $H^2(\til M,\R)=0$, we observe that both symplectic forms $F_{\pm}=-2g(I\pm J)^{-1}$ cannot be defined entirely on $\til M$, hence both sets
	\[
	Z_{\pm}=\{x\in \til M\ |\ (I\pm J)_x=0\}
	\]
	are nonempty distinct proper analytic subsets.
	
	Let $\sigma_I\in \Gamma(\Lambda^2(T^{1,0}\til M))$ be the corresponding holomorphic Poisson tensor given by
	\[
	\sigma_I = \sigma-\sqrt{-1}I\sigma,\quad \sigma=\frac{1}{2}[I,J]g^{-1}.
	\]
	Then $\sigma_I$ vanishes on $Z_+\cup Z_-$, yielding a disconnected anticanonical divisor on $(\til M,I)$.  Using the same argument as in the proof of Proposition~\ref{prop:hopf_Q_sigma}, we observe that any holomorphic bivector on a non-diagonal Hopf surface (with $\lambda\neq 0$) is a complex multiple of $z_1^2\frac{\del}{\del z_1}\wedge \frac{\del}{\del z_2}$. Thus 
	the non-diagonal primary Hopf surfaces and their quotients do not admit disconnected anticanonical divisors, and cannot carry an even GK structure. Therefore $(\til M,I)$ is biholomorphic to a diagonal Hopf surface $(\til M,I_{\alpha\beta})$.
	
	It remains to figure out which free finite quotients of $(\til M,I_{\alpha\beta})$ admit an even GK structure. Let $\Gamma\subset \mathrm{Aut}(\til M,I_{\alpha\beta})$ be a finite group acting freely, and assume that the quotient $M=\til M/\Gamma$ endowed with the inherited complex structure $I_{\alpha\beta}$ admits an even GK structure $(M,g,I_{\alpha\beta},J)$. By abuse of notation, let $(\til M,g,I_{\alpha\beta},J)$ be the lift of this GK structure to $\til M$. Then $\Gamma$ acts biholomorphically on $(\til M,I_{\alpha\beta})$ preserving $g,I_{\alpha\beta}$ and $J$. In particular $\Gamma$ preserves the underlying Poisson tensor $\sigma_I$. By the first half of Proposition~\ref{prop:hopf_Q_sigma}, there is a biholomorphism
	\[
	\phi \colon (\til M,I_{\alpha\beta})\to (\til M, I_{\alpha\beta})
	\]
	such that $\phi_*\sigma_I=c\pmb\pi$, where $c\in\C^*$ is a nonzero constant and $\pmb\pi=z_1z_2 \frac{\del}{\del z_1}\wedge \frac{\del}{\del z_2}$. The group
	\[
	\Gamma_\phi=\{\phi\circ \gamma\circ\phi^{-1}\ |\ \gamma\in\Gamma\}
	\]
	preserves $\pmb\pi$, thus by the second half of Proposition~\ref{prop:hopf_Q_sigma}, $\Gamma_\phi$  is a subgroup of $(\C^*)^2/\la A\ra\subset \mathrm{GL_2}(\C)/\la A\ra$. Let $\til{\Gamma_{\phi}}$ be the lift of $\Gamma_{\phi}$ to $(\C^*)^2\subset \mathrm{GL}_2(\C)$. Following~\cite[(5.8)]{Bea} we observe that $\til{\Gamma_{\phi}}$ acts freely on $\C^2\backslash\{\mathbf{0}\}$, hence the projections $\til \Gamma_{\phi}\to (\C^*\times \{\mathbf 1\})$ and $\Gamma_{\phi}\to (\{\mathbf 1\}\times \C^*)/\la A\ra$ are injective. We deduce that $\til\Gamma_{\phi}$ is isomorphic to a subgroup of $\C^*$ acting freely and co-compactly on $\C^2\backslash\{\mathbf{0}\}$. Any such group $\til\Gamma_{\phi}$	
	must be of the form
	\[
	\til\Gamma_\phi=\Z_\ell\times \Z,
	\]
	where $\Z_\ell$ acts via multiplication by primitive $\ell^{\mathrm{th}}$ roots of 1:
	\[
	(z_1,z_2)\mapsto (\exp{\tfrac{2\pi \sqrt{-1}k}{\ell}}z_1, \exp{\tfrac{2\pi \sqrt{-1}}{\ell}}z_2)
	\]
	and $\Z$ acts via multiplication by a matrix $B=\left(\begin{matrix}
	\alpha_0 & 0 \\ 0 & \beta_0
	\end{matrix}\right)$ such that $\alpha_0,\beta_0\in\C$, $0<|\alpha_0|, |\beta_0|<1.$ 
	
	Thus the initial secondary Hopf surface $(M,I)$ is the $\Gamma$-quotient of $(\til M,I_{\alpha\beta})$ which is biholomorphic (via $\phi$) to the $\Gamma_\phi$-quotient of $(\til M,I_{\alpha\beta})$. Given the description of possible groups $\Gamma_{\phi}$, the latter is the $\Gamma_{k,\ell}\simeq\Z_\ell$-quotient of $(\C^2\backslash\{\mathbf 0\})/\la B\ra$, noting that  $B$ generates a $\mathbb Z$-action on $\mathbb C^2 \backslash(\mathbf 0\}$.  We conclude that $(M,I)$ is necessarily a $\Gamma_{k,\ell}$-quotient of a diagonal Hopf proving that only $\Gamma_{k,\ell}$-cyclic quotients of primary diagonal Hopf surfaces could admit even GK structures.
	
	Conversely, the construction of~\cite{st-us-19} yields an $(S^1)^3$-invariant even GK structure on every diagonal Hopf surface, which thus descends to all cyclic $\Z_\ell$-quotients, acting as above. This provides an even GK structure on the secondary Hopf surface $(\til M,I_{\alpha\beta})/\Gamma_{k,\ell}$.
\end{proof}

It is also interesting to determine which holomorphic Poisson tensors $\pmb\pi$ on $(M,I_{\alpha\beta})$ correspond to an even GK structure $(M,g,I_{\alpha\beta},J)$ with $\sigma:=\frac{1}{2}[I,J]g^{-1}=\Re(\pmb\pi)$. To this end, let us first state some basic facts about complex geometry of diagonal Hopf surfaces. Every Hopf surface $(M,I_{\alpha\beta})$ presented as $\C^2\backslash\{\mathbf{0}\}/\la A\ra$ has a distinguished holomorphic bivector
\begin{equation}\label{eq:hopf_poisson}
\pmb\pi:=z_1z_2\frac{\partial}{\partial z_1}\wedge \frac{\partial}{\partial z_2}
\end{equation}
such that its zero locus is the standard disconnected anticanonical divisor $E_1\cup E_2$:
\[
E_1=\{z_1=0\}/\la A\ra,\quad E_2=\{z_2=0\}/\la A\ra.
\]
Observe that $E_1\simeq\C^*/\la \alpha\ra$ and $E_2\simeq\C^*/\la \beta\ra$. Since $\dim_\C M=2$, the holomorphic $(2,0)$-vector $\pmb\pi$ is automatically Poisson. If $\sigma'_{I_{\alpha\beta}}$ is another nontrivial section of $\Lambda^2(T^{1,0}M)$ then the zero locus of $\sigma'_{I_{\alpha\beta}}$ is either an elliptic curve with multiplicity $\geq 2$ (this option is possible only when $\alpha^q=\beta$ or $\beta^q=\alpha$ for some $q\in \Z$), or again is a union of two elliptic curves $E_1'\cup E_2'$. From Proposition~\ref{prop:hopf_Q_sigma} in the latter case, there exists
\begin{equation}\label{eq:divisor_pullback}
\gamma\in \mathrm{Aut}(M,I_{\alpha\beta})\quad\mbox{such that}\quad\gamma(\{E_1,E_2\})=\{E_1',E_2'\}.
\end{equation}
Let $\mathrm{Aut}(M,I_{\alpha\beta},\pmb\pi)\subset \mathrm{Aut}(M,I_{\alpha\beta})$ be the subgroup of  automorphisms preserving $\pmb\pi$, which by Proposition~\ref{prop:hopf_Q_sigma}(ii) is isomorphic to $(\C^*)^2/\la A\ra$.
Observe that the group $\mathrm{Aut}(M,I_{\alpha\beta},\pmb\pi)\simeq (\C^*)^2/\la A\ra$ contains a unique maximal torus $K\simeq (S^1)^3$ with  Lie algebra $\mf k$:
\begin{equation} \label{f:Kdef}
\begin{split}
K&\subset \mathrm{Aut}(M,I_{\alpha\beta},\pmb\pi),\\
\mf k &= 
\mathrm{span}_\R\left(
\Im(z_1\frac{\partial}{\partial z_1}),\ 
\Im(z_2\frac{\partial}{\partial z_2}),\ 
\log|\alpha|\Re(z_1\frac{\partial}{\partial z_1})+\log|\beta|\Re(z_2\frac{\partial}{\partial z_2})
\right).
\end{split}
\end{equation}
This group will play an important role below in this section.

%

We close this subsection by showing a key geometric property of even generalized K\"ahler structures on Hopf surfaces. If $(M,g,I_{\alpha\beta},J)$ is an even GK structure, it determines a holomorphic Poisson tensor $\sigma_{I_{\alpha\beta}}$ (see~\eqref{eq:poisson_holo}). It turns out that up to action of  $\mathrm{Aut}(M,I_{\alpha\beta})$, $\sigma_{I_{\alpha\beta}}$ is actually a  \emph{real} multiple of $\pmb{\pi}$.  This apparent rigidity in the choice of the holomorphic Poisson tensor is in contrast to the case of GK structures on K\"ahler backgrounds.  We discuss this phenomenon further in \S \ref{s:SOGKS}.

\begin{thm}\label{thm:hopf_poisson_gk}
	Let $(M,g,I_{\alpha\beta},J)$ be an even generalized K\"ahler on a diagonal Hopf surface $(M,I_{\alpha\beta})$. Then up to the action of $\mathrm{Aut}(M,I_{\alpha\beta})$ the holomorphic Poisson tensor 
	\[
	\sigma_{I_{\alpha\beta}}=\sigma-\sqrt{-1}{I_{\alpha\beta}}\sigma,\quad \sigma:=\frac{1}{2}[I_{\alpha\beta},J]g^{-1}
	\]
	is given by
	\[
	\sigma_I=c\cdot z_1z_2\frac{\del}{\del z_1}\wedge\frac{\del}{\del z_2},
	\]
	where $c\in\R$ is a nonzero real number.
\end{thm}
\begin{proof}
	For convenience, throughout this proof we denote $I=I_{\alpha\beta}$. 
	As in the proof of Theorem~\ref{thm:hopf_existence_even}, 
	$Z_\pm=\{x\in M\ |\ (I\pm J)_x=0\}$ are two nonempty disjoint  components of the zero locus of $\sigma_I$. Since $\mathrm{Aut}(M,I)$ acts transitively on disconnected canonical divisors on $M$, after applying appropriate automorphism, we can assume that
	\[
	\{\sigma_I=0\}=\{\pmb\pi=0\}.
	\]
	Hence $\sigma_I$ is a constant \emph{complex} multiple of $\pmb\pi$. After scaling the metric $g$ and correspondingly the Poisson tensor $\sigma_I$, we can assume that
	\[
	\sigma_I=e^{\sqrt{-1}s}\pmb\pi,\quad s\in \R.
	\]
	Our goal is to prove that the \emph{phase} factor $e^{\sqrt{-1}s}$ must be real.
	
	Without loss of generality assume that $I+J=0$ on the elliptic curve $E_1:=\{z_1=0\}\subset M$ and $I-J=0$ on $E_2:=\{z_2=0\}\subset M$. Define $p=-\tfrac{1}{4}\tr(IJ)$. The smooth function $p$ is called the \emph{angle function}, and takes its values in $[-1,1]$, such that
	\[
	p^{-1}(-1)=E_1,\quad p^{-1}(1)=E_2.
	\]
	A general relation on an even nondegenerate 4-dimensional GK manifold~\cite[Lemma\,3.8]{ap-st-17} states that
	\begin{equation}\label{eq:p_sigma}
	\sigma \left(\frac{2dp}{1-p^2}\right)=\theta_I^\sharp-\theta_J^\sharp.
	\end{equation}
	Furthermore, we necessarily have $\theta_I=*d^c_I\omega_I=*H=-*d^c_J\omega_J=-\theta_J$. Using this observation and the relation $|\sigma_I|^2_g=(1-p^2)$, we can rewrite the above identity as
	\begin{equation}\label{eq:lee_canonical}
	\sigma(d\log|\sigma|^2_g)=-2p\theta_I^\sharp.
	\end{equation}
	Now, if we change metric $g$ on the line bundle $-K_M$ to any other Hermitian metric $g_{\varphi}:=e^\varphi g$, $\varphi\in C^\infty(M,\R)$ then the left hand side of~\eqref{eq:lee_canonical} will change by a summand $\sigma(d\varphi)$. Since $\sigma$ vanishes on $E_1\cup E_2\subset M$, it implies that the restriction of $p\theta_I^\sharp$ to the elliptic curves $E_i$ is uniquely determined by $\sigma_I$.
	Specifically, from~\eqref{eq:p_sigma} it follows that
	\[
	\theta_I^\sharp\Big|_{\til E_1}=\Re(e^{\sqrt{-1}s}z_2\frac{\del}{\del z_2}),\quad
	\theta_I^\sharp\Big|_{\til E_2}=\Re(e^{\sqrt{-1}s}z_1\frac{\del}{\del z_1})
	\]
	We are going to prove that $e^{\sqrt{-1}s}$ is real, so $\theta_I^\sharp$ restricted to $\{z_2=0\}$ equals $\pm\Re(z_1\frac{\del}{\del z_1})$ and being restricted to $\{z_1=0\}$ equals $\pm\Re(z_2\frac{\del}{\del z_2})$.	
	
	Let $\til E_i$ be the preimages of the elliptic curves $E_i$ in the universal cover $\til M$. Pick a generator $\gamma\in \pi_1(M)$ of the deck transformation group, and choose identifications $\zeta_2\colon (\til E_1,I)\simeq \C^*$ and $\zeta_1\colon (\til E_2,I)\simeq \C^*$ in such a way that $\gamma$ acts by multiplication with $\alpha$ and $\beta$ respectively. Then there exists a complex number $\xi\in\C$ such that, depending on the choice of the generator $\gamma\in\pi_1(M)$, either $\zeta_1=\xi z_1$ and $\zeta_2=\xi z_2$, or $\zeta_1=\xi z_1^{-1}$ and $\zeta_2=\xi z_2^{-1}$. Then either vector field $\theta_I^\sharp$ restricted to divisors $\til E_{1}$ and $\til E_{2}$ is given by $\Re(e^{\sqrt{-1}s}\zeta_i\frac{\del}{\del\zeta_i})$, or it is given by $\Re((-e^{\sqrt{-1}s})\zeta_i\frac{\del}{\del\zeta_i})$. In both cases, we have the following identity of complex numbers
	\begin{equation}\label{eq:hopf_poisson_gk_pf1}
	\frac{(\theta_I^\sharp)^{1,0}_I\Big|_{\til E_1}}{\zeta_2\,\frac{\del}{\del\zeta_2}}=\frac{(\theta_I^\sharp)^{1,0}_I\Big|_{\til E_2}}{\zeta_1\,\frac{\del}{\del\zeta_1}}
	\end{equation}
	
	Now, consider the same smooth manifold $M$ with respect to the other complex structure $J$. We will have $J=-I$ on $E_1=\{z_1=0\}$ and $J=I$ on $E_2=\{z_2=0\}$. We get holomorphic identifications:
	\[
	\bar\zeta_2\colon (E_1,J)\simeq \C^*,\quad \zeta_1\colon (E_2,J)\simeq \C^*.
	\]
	By the classification of complex structures on $S^3\times S^1$, see~\cite{kod-66}, $(M,J)$ is necessarily biholomorphic to a diagonal Hopf surface with parameters $\alpha$ and $\bar\beta$ being determined by the moduli of elliptic curves. Repeating the same calculation as above with $(M,J)$ 
	we get
	\begin{equation}\label{eq:hopf_poisson_gk_pf2}
	\frac{(\theta_J^\sharp)^{1,0}_J\Big|_{\til E_1}}{\bar\zeta_2\,\frac{\del}{\del\bar\zeta_2}}=\frac{(\theta_J^\sharp)^{1,0}_J\Big|_{\til E_2}}{\zeta_1\,\frac{\del}{\del\zeta_1}}.
	\end{equation}
	Now we observe that since $I=-J$ on $E_1$ and $\theta_I^\sharp=-\theta_J^\sharp$, the expressions on the left hand side of~\eqref{eq:hopf_poisson_gk_pf1} and~\eqref{eq:hopf_poisson_gk_pf2} are negative conjugate of each other
	\[
	\frac{(\theta_I^\sharp)^{1,0}_I\Big|_{\til E_1}}{\zeta_2\,\frac{\del}{\del\zeta_2}}=-\frac{(\theta_J^\sharp)^{1,0}_I\Big|_{\til E_1}}{\zeta_2\,\frac{\del}{\del\zeta_2}}=
	-\frac{\bar{(\theta_I^\sharp)^{1,0}_J\Big|_{\til E_1}}}{\zeta_2\,\frac{\del}{\del\zeta_2}}=
	-\bar{\left(\frac{(\theta_I^\sharp)^{1,0}_J\Big|_{\til E_1}}{\bar\zeta_2\,\frac{\del}{\del\bar\zeta_2}}\right)}
	\]
	At the same time, the expressions on the right hand side of~\eqref{eq:hopf_poisson_gk_pf1} and~\eqref{eq:hopf_poisson_gk_pf2} are just negative of each other, as $I=J$ on $E_2$.  Thus both must be real, forcing $e^{\sqrt{-1}s}$ to be $\pm 1$.
	\end{proof}

\subsection{Uniqueness of GKRS on Hopf surfaces}

Since we know that any diagonal Hopf surface is a part of a GK structure, it is next important to identify whether the set of all GK structures extending $(M,I_{\alpha\beta})$ has a distinguished representative. In~\cite{st-us-19} we established the following existence result for gradient steady GKRS:

\begin{thm}[{\cite[Theorem\,1.1]{st-us-19}}]\label{thm:hopf_soliton_existence}
	Let $(M, I_{\alpha\beta})$ be a diagonal 
	Hopf surface. Then there exists a Hermitian metric $g^s$ and $g^s$-orthogonal complex structures $J^s_1$ and $J^s_2$,  such that
	\begin{enumerate}
		\item $(M, g^s, I_{\alpha\beta}, J^s_1)$ is an odd-type gradient steady GKRS with $(M,J^s_1)$ biholomorphic to $(M,I_{\alpha\bar{\beta}})$.
		\item $(M, g^s, I_{\alpha\beta}, J^s_2)$ is an even-type gradient steady GKRS with $(M,J^s_2)$ biholomorphic to $(M,I_{\alpha\bar{\beta}})$.
	\end{enumerate}
	Moreover $g^s$ and $J_{i}^s$ are invariant under the maximal compact subgroup $K\subset \mathrm{Aut}(M,I_{\alpha\beta}, \pmb{\pi})$.
\end{thm}

Since the GK solitons of Theorem~\ref{thm:hopf_soliton_existence} are invariant under the group $K\subset \mathrm{Aut}(M,I_{\alpha\beta},\pmb\pi)$, each of them descends to a GK soliton on the secondary Hopf surfaces given by the quotient of $(M,I_{\alpha\beta})$ by a cyclic group $\Z_{\ell}\simeq \Gamma_{k,\ell}\subset K$ acting via multiplication by primitive roots of 1
\[
(z_1,z_2)\mapsto (\exp{\tfrac{2\pi \sqrt{-1}k}{\ell}}z_1, \exp{\tfrac{2\pi \sqrt{-1}}{\ell}}z_2).
\]
In the view of Theorem~\ref{thm:hopf_existence_even}, we see that every Hopf surface (primary or secondary) admitting a GK structure, also admits a gradient steady GK soliton. The purpose of this section is to establish the uniqueness counterpart of the existence statement, proving that these solitons are unique modulo the action of the automorphism group $\mathrm{Aut}(M,I_{\alpha\beta})$.

\begin{thm}\label{thm:hopf_soliton_uniqueness}
	Let $(M, I_{\alpha\beta})$ be a diagonal Hopf surface. Assume that there exists a gradient steady GKRS $(M,g,I_{\alpha\beta},J)$. Then
	\begin{enumerate}
		\item $g$ is invariant under the maximal compact subgroup $K\subset \mathrm{Aut}(M,I_{\alpha\beta}, \pmb{\pi})$;
		
		\item There exists an element $\gamma\in \mathrm{Aut}(M,I_{\alpha\beta})$ and a constant $c>0$ such that
		\[
		\gamma^*g=c\cdot g^s,\quad \gamma^*J=\pm J_i^s,
		\]
		where $g^s$ and $J_i^s$ are the Hermitian metric and the complex structures of Theorem~\ref{thm:hopf_soliton_existence}.
	\end{enumerate}
\end{thm}

Before proving Theorem \ref{thm:hopf_soliton_uniqueness}, we show two preliminary technical lemmas.  The first is a general fact stating that four-dimensional even generalized K\"ahler structures are determined uniquely by the metric, one of the complex structures, and the Poisson tensor $\gs$, unless the structure is hyperK\"ahler.

	\begin{lemma} \label{lm:4DGK}
		Given a connected 4-dimensional Hermitian manifold $(M^4,g,I)$ which is not hyperK\"ahler and a real Poisson tensor $\sigma \neq 0$, then there exists at most one complex structure $J$ extending $(M,g,I)$ to an even generalized K\"ahler structure with $\sigma=\frac{1}{2}[I,J]g^{-1}$.
	\end{lemma}
	\begin{proof} As $\sigma\neq 0$, any $J$ satisfying the hypothesis will define a GK structure $(g, I, J)$ of even type, and furthermore $J\neq \pm I$.
	Notice that the tensors $g$ and $\sigma$ uniquely determine the tensor $[I,J]$.  At a given point $x\in M$,  there is an $S^2$-worth of complex structures compatible with $g$ and the orientation induced by  $I$.  Let $p:=-1/4\tr(IJ)$ be the angle function. In the 4-dimensional case we have $IJ+JI=-2p\,\mathrm{Id}$, thus the triple $p,I,[I,J]$ uniquely determines $J$, so it remains to identify $p$. Since $|[I,J]|_g=1-p^2$, the metric $g$ and Poisson tensor $\sigma$ determine $p$ up to sign. There is also a general relation on a nondegenerate GK $4$-manifold~\eqref{eq:p_sigma}:
		\begin{equation*}
		\sigma \left(\frac{2dp}{1-p^2}\right)=\theta_I^\sharp-\theta_J^\sharp.
		\end{equation*}
		Since on an even GK 4-manifold $\theta_I=*_gH=-\theta_J$, and the subset of points where  $\sigma=\frac{1}{2}[I, J]g^{-1}=0$   is an analytic subset of $(M, I)$, this allows to recover $dp$ from $\sigma,g,I$. In other words, if $J_1$ and $J_2$ are two different complex structures satisfying the hypothesis of the lemma, the corresponding angle functions $p_1$ and $p_2$ must defer by an additive constant and satisfy $p_1 =  \pm p_2$ at any point.  If $p_1 = - p_2$ everywhere, then $d p_i = 0$, and then $(I, J_1, J_2)$ determine a hyperK\"ahler structure compatible with $g$, a contradiction.  Thus there is a point where $p_1 = p_2$, and it then follows that $p_1 \equiv p_2$, thus $J_1\equiv J_2$ by the argument above.
\end{proof}

Next, in the course of the proof we will show that $(M,I,\sigma_I)$ and $(M,J,\sigma_J)$ have explicit real holomorphic vector fields $Y_I$ and $Y_J$ preserving $\sigma=\Re(\sigma_I)=\Re(\sigma_J)$, admitting Hamiltonian potentials on $U_{\pm}$ with logarithmic singularities along $M\backslash U_{\pm}$. We will further show that the difference of such potentials extends to a smooth function on $M$, and then Lemma~\ref{lm:vf_vanish} below will yield that $Y_I=Y_J$, and we thus obtain a jointly $I$- and $J$-holomorphic Killing vector field. This is the key step where the variational approach to GK solitons and the weighted $\pmb J$-functional of Section~\ref{sec:aubin} enter the proof of Theorem~\ref{thm:hopf_soliton_uniqueness}.

\begin{lemma}\label{lm:vf_vanish}
	Let $(M,g,I,J)$ be a compact log-nondegenerate gradient steady GKRS. Denote by $G\subset \mathrm{Aut}(g,I,J)$ the group generated by the soliton vector fields $X_I,X_J$.  Suppose there exist $G$-invariant vector fields $Y_I, Y_J$ preserving $\gs$ such that $Y_I$ is $I$-holomorphic, $Y_J$ is $J$-holomorphic, and $Y_J - Y_I$ is $\gs$-Hamiltonian.  Then $Y_I = Y_J$.
\end{lemma}
\begin{proof}
	Let $\Psi_{Y_J,t}$ be the one-parameter family of diffeomorphisms generated by $Y_J$, and consider the one-parameter family of steady gradient GKRS $(M,g_t,I_t,J_t)$ given by pullback by $\Psi_{Y_J,t}$, i.e.
	\begin{equation}\label{eq:pf_pullback_yj}
	g_t = \Psi_{Y_J,t}^* g, \quad \sigma_t= \Psi_{Y_J,t}^* \sigma = \gs,\quad I_t=\Psi_{Y_J,t}^*I,\quad J_t=\Psi_{Y_J,t}^*J=J,
	\end{equation}
        where we have used that $Y_J$ is $J$-holomorphic and preserves $\gs$.  By assumption $Y_J - Y_I$ is $\gs$-Hamiltonian, and so we can express
	\begin{align*}
	Y_J - Y_I = X_{\fp}.
	\end{align*}
	Now let $\fp_t=\Psi_{Y_J,t}^*\fp$, and $Y_{I_t}:= \Psi_{Y_J,t}^*Y_I$ be the pull backs under the diffeomorphism $\Psi_{Y_J,t}$. Then, since $Y_J$ preserves $\sigma$, for any $t$ we have
	\[
	Y_J-Y_{I_t} = X_{\fp_t}.
	\]
	
	We claim that the family of GK structures $(M,g_t,I_t,J)$ is also given by the Hamiltonian flow construction with the family of potentials $\fp_t$. 
	Indeed, along the flow construction we have
	\[
	\frac{d}{dt} I_t=\mc L_{X_{\fp_t}}I=\mc L_{Y_J-Y_{I_t}}I_t=\mc L_{Y_J}I_t,
	\]
	since $Y_{I_t}$ is $I_t$-holomorphic.  Thus the evolution of $I_t$ under the flow construction generated by $\fp_t$ coincides with the deformation given by~\eqref{eq:pf_pullback_yj}. As the metric $g_t$ is uniquely determined by $\gs$, $I$, and $J$, it follows that $g_t$ also evolves according to the flow construction as in (\ref{f:flowconst}).  Thus the family $(M,g_t,I_t,J_t)$ of gradient steady GKRS lies in $\mc M^G$.  We can now apply the rigidity statement of Theorem~\ref{thm:rigid_soltions} and conclude that $\fp_t$ is constant, so $X_{\fp}=0$ and $Y_I = Y_J$.
\end{proof}

\begin{rmk}
	There is no non-trivial analogue of this result in the $\sigma$-nondegenerate setting, since $(M,I,\sigma_I^{-1})$ is holomorphic symplectic, and its automorphism group $\mathrm{Aut}(M,I,\sigma)$ is discrete. This shows that the modification of the $\pmb J$-functional of~\cite{ap-st-17} to the log-nondegenerate setting is essential.
\end{rmk}

\begin{proof}[Proof of Theorem~\ref{thm:hopf_soliton_uniqueness}]
	Let $(M,g,I_{\alpha\beta},J)$ be a gradient steady GKRS on a diagonal Hopf surface $(M,I_{\alpha\beta})$. For notational convenience, throughout this proof we denote
	\[
	I=I_{\alpha\beta}.
	\]
	There are two essentially different cases. First, the case of $(M,g,I,J)$ being of odd type (i.e. $I$ and $J$ induce opposite orientations). In this case $I$ and $J$ commute, and operator $Q=IJ$ defines a holomorphic splitting of $T^{1,0}M$ into subbundles $L_1=\ker(\mathbf{1}-Q)$ and $L_2=\ker(\mathbf{1}+Q)$. By Proposition~\ref{prop:hopf_Q_sigma}, there is a biholomorphism $\phi\colon (M,I)\to (M,I)$ such that $\phi^*Q=Q_{\mathrm{std}}$ is the standard holomorphic splitting. In this case it is proved in~\cite[\S 5]{st-us-19} that $(M,g,I,J)$ is unique up to the action of the automorphism group $\mathrm{Aut}(M,I)$.
	
	Thus from now on we assume that $(M,g,I,J)$ is of even type. As in the proof of Theorem~\ref{thm:hopf_existence_even}, the corresponding holomorphic Poisson tensor
	\[
	\sigma_I = \sigma-\sqrt{-1}I\sigma,\quad \sigma=\frac{1}{2}[I,J]g^{-1}
	\]
	vanishes on $Z_+\cup Z_-$, which is a disjoint union of two elliptic curves in $(M,I)$. By Proposition~\ref{prop:hopf_Q_sigma}, we can find an automorphism $\gamma\in\mathrm{Aut}(M,I)$ such that $\gamma(Z_+\cup Z_-)=E_1\cup E_2$, where $E_i=\{z_i=0\}/\la A\ra$ are the standard elliptic curves in $(M,I)$. Furthermore by Theorem~\ref{thm:hopf_poisson_gk}, $\gamma^*\sigma_I$ is a \emph{real} multiple of $\pmb{\pi}$.  Hence we pull back the entire structure $(M,g,I,J)$ via $\gamma$, and change $J$ to $-J$, if necessary, so that $\sigma_I=\lambda\pmb\pi$,  $\lambda\in\R$ and
	\[
	Z_+=\{x\in M\ |\ (I+J)_x=0\}=E_1=\{z_1=0\}/\la A\ra
	\]
	\[
	Z_-=\{x\in M\ |\ (I-J)_x=0\}=E_2=\{z_2=0\}/\la A\ra
	\]
	
	
	With the use of Lemma~\ref{lm:4DGK}, to prove the theorem, it suffices to uniquely determine $g$ and $\sigma$. In~\cite{st-us-19} we proved the uniqueness of a $K$-invariant Hermitian metric $g$ solving the soliton system~\eqref{eq:soliton_system} modulo scalings and the action of $\mathrm{Aut}(M,I,\pmb\pi)$,
	where  $K$ is the maximal compact group of automorphisms given by (\ref{f:Kdef}). Thus to prove the uniqueness of $g$ it remains to prove that the soliton metric $g$ must be invariant under the action of $K$.
	
	As before, let
	\[
		X_I:=\frac{1}{2}I(\theta_I^\sharp-\nabla f),\quad X_J:=\frac{1}{2}J(\theta_J^\sharp-\nabla f).
	\]
	be the vector fields associated with the soliton $(M,g,I,J)$. It follows from the discussion in Section~\ref{sec:gks} that $X_I+X_J$ is $F_+$-Hamiltonian on $U_+=M\backslash Z_+$ with the Hamiltonian potential $\psi_+=-f+\frac{1}{2}\log\det(I+J)$. Since $M$ is compact, and $\psi_+$ goes to $-\infty$ near $Z_+$, function $\psi_+$ attains a maximum, ensuring that vector field $X_I+X_J$ vanishes at some point in $U_+$, but not identically. The same argument shows that $X_I-X_J$ vanishes at some point in $U_-$.
	
	Since $X_I+X_J$ belongs to $\mf k$ and has a zero on $U_+$, it could only be a nonzero multiple of $\Im(z_2\frac{\partial}{\partial z_2})$, similarly $X_I-X_J$ must be a nonzero multiple of $\Im(z_1\frac{\partial}{\partial z_1})$. Hence
	\[
	\mathrm{span}_\R(X_I,X_J)=\mathrm{span}_\R\left(
	\Im(z_1\frac{\partial}{\partial z_1}),\ 
	\Im(z_2\frac{\partial}{\partial z_2})
	\right)
	\]
	is a real 2-dimensional subspace in $\mf k$. To prove that $g,I,J$ are invariant under that action of $K\simeq (S^1)^3$, it remains to find a third vector field in $\mf k$, preserving $I,J$ and $g$.
	
	Since $\sigma_I$ is a multiple of $z_1z_2\frac{\partial}{\partial z_1}\wedge \frac{\partial}{\partial z_2}$, there exists a vector field (here even a complex multiple would suffice), there exists a vector field
	\[
	Y_I \in \mathrm{span}_{\R} \left\{ \log|\alpha|\Im(z_1\tfrac{\partial}{\partial z_1})+ 
	\log|\beta|\Im(z_2\tfrac{\partial}{\partial z_2}),\ 
	\log|\alpha|\Re(z_1\tfrac{\partial}{\partial z_1})+\log|\beta|\Re(z_2\tfrac{\partial}{\partial z_2}) \right\} \subset \mf k
	\]
	such that $Y_I$ preserves $\sigma_I$ and has $\Re(\sigma_I)$-Hamiltonian potential on $U_+\cap U_-=M\backslash(Z_1\cup Z_2)$
	\[
	h_{Y_I}=\log|\alpha|\log|z_2|-\log|\beta|\log|z_1|.
	\]
	Now, we can swap the roles of $I=I_{\alpha\beta}$ and $J$ and remember that $(M,J)$ must be also a Hopf surface with parameters $(\alpha,\bar\beta)$ (which is determined by the moduli of the elliptic curves $E_1$ and $E_2$, keeping in mind that $I_{\alpha\beta}=-J$ on $E_1$ and $I_{\alpha\beta}=J$ on $E_2$). There is a $J$-holomorphic Poisson tensor
	\[
	\sigma_J=\Re(\sigma_I)-\sqrt{-1}J\Re(\sigma_I)
	\]
	on $(M,J)$. Similarly, we find a vector field $Y_J$ in the Lie algebra of the maximal compact subgroup $(S^1)^3$
	of $\mathrm{Aut}(M,J,\sigma_J)\simeq (\C^*)^2/\Z$, which preserves $\sigma_J$ and has a Hamiltonian potential
	\[
	h_{Y_J}=\log|\alpha|\log|w_2|-\log|\beta|\log|w_1|,
	\]
	where $w_1,w_2$ are holomorphic coordinates of the universal cover $(\widetilde{M},J)\simeq \C^2\backslash\{\mathbf{0}\}$ such that $\sigma_J$ is proportional to $w_1w_2\frac{\partial}{\partial w_1}\wedge \frac{\partial}{\partial w_2}$. 
	
	We claim that the vector field $Y_I-Y_J$ has a global smooth $\sigma$-Hamiltonian potential on $M$. We will prove that locally defined functions $\log|z_1|-\log|w_1|$ and $\log|z_2|-\log|w_2|$ extend to smooth functions across the set $Z_{\pm}$. Indeed, since $\sigma_I$ is a holomorphic section of the anticanonical bundle with the simple zero along the divisor $E_1\cup E_2$, section $\sigma_I/z_1$ extends to a local nonvanishing section in a neighbourhood of $\{z_1=0\}$.
	
	Analogously $\sigma_J/w_1$ is a holomorphic nonvanishing section in a neighbourhood of $\{w_1=0\}$. It remains to note that if we choose the metrics on $K_{M,I}$ and $K_{M,J}$ induced by the Hermitian metric $g$, then $|\sigma_J|^2_g=|\sigma_I|^2_g$, since $\sigma_I$ and $\sigma_J$ have a common real part. Therefore
	\[
	\frac{|w_1|^2}{|z_1|^2}=\frac{|\sigma_I/z_1|^2_g}{|\sigma_J/w_1|^2_g}
	\]
	is a well-defined nonvanishing function in a neighbourhood of a point on $\{z_1=0\}=\{w_1=0\}$. The same argument works for $|w_2|^2/|z_2|^2$, and the claim is proved.
	
	Now we can apply Lemma~\ref{lm:vf_vanish} to the gradient steady GKRS $(M,g,I,J)$ and vector fields $Y_I$ and $Y_J$, concluding that $Y_I=Y_J$. It gives us 3 vector fields $X_I,X_J,Y_I$ which preserve $I$, $J$ and $g$. It remains to prove that they are linearly independent so that
	\[
	\mathrm{span}_\R(X_I,X_J,Y_I)=\mf k.
	\]
	Assume on the contrary that $Y_I$ lies in the span of $X_I,X_J$.  Using part (7) of Proposition~\ref{pr:XI_XJ_properties} and the fact that $\sigma$ is $I$- and $J$-anti-invariant, we observe that if $Y_I$ were in the span of $X_I$ and $X_J$, then
	\[
	dh_{Y_I}(IX_I) = \sigma^{-1}(IX_I,Y_I)=0,\qquad dh_{Y_I}(JX_J) = \sigma^{-1}(IX_J,Y_I)=0.
	\]
	Together these imply, using that $I(X_I + X_J)$ is a multiple of $\Re(z_1\frac{\del}{\del z_1})$,
	\[
	0=dh_{Y_I}(I(X_I+X_J))= I(X_I+X_J) \cdot \left(\log|\alpha|\log|z_2|-\log|\beta|\log|z_1|\right)= \mathrm{const} \cdot \log|\beta|,
	\]
	where the constant is nonzero, since $X_I+X_J$ is a nonzero vector.  This is a contradiction, and thus $\mathrm{span}_{\R}(X_I,X_J,Y_I) = \mf k$, and we conclude that the GKRS $(M,g,I,J)$ must be invariant under the maximal compact subgroup $K\subset \mathrm{Aut}(M,I_{\alpha\beta},\pmb\pi)$.
	By~\cite[Theorem~\,1.1]{st-us-19}, the metric $g$ is uniquely defined up to a constant multiple and the action of $\mathrm{Aut}(M,I_{\alpha\beta},\pmb\pi)$.
	
	Let us fix such metric $g$. It remains to identify the Poisson tensor $\sigma$. We know that $\sigma_I=\sigma-\sqrt{-1}I\sigma$ is a real multiple of $\pmb\pi$, so that
	\[
	\sigma = \lambda\Re\left(z_1z_2\frac{\del}{\del z_1}\wedge \frac{\del }{\del z_2}\right),\quad \lambda\in \R.
	\]
	It remains to determine possible values of the constant $\lambda$. Since on a GK manifold we have
	\[
	|\sigma|^2_g=1-p^2,
	\]
	$g$ is already determined (up to scaling), and $p$ takes values between $-1$ on $Z_-$ and $+1$ on $Z_+$, we have that $\lambda$ is uniquely determined up to a sign by the condition that
	\[
	\sup_{x\in M} |\sigma_x|^2_g=1. 
	\]
	By the claim in the beginning of the proof, for each such choice of $\lambda$ there is at most one GKRS with the prescribed $g,I,\sigma$. Our convention above on $J$ assumes that $I+J=0$ on $\{z_1=0\}$ and $I-J=0$ on $\{z_2=0\}$, which prescribes uniquely the angle function $p$. Thus the sign of $\lambda$ is uniquely determined by the identity~\eqref{eq:p_sigma}. For the right choice of the sign of $\lambda$ there is indeed a genuine GKRS $(g,I,J,\sigma)$ as constructed in~\cite[\S 3.2]{st-us-19}. 
	The change of the complex structure $J\mapsto -J$ would result into the change of the signs of the angle function $p$ and of $\lambda$. 
	
	The uniqueness of even GKRS on a diagonal Hopf surface $(M,I_{\alpha\beta})$ is proved.
	
\end{proof}

\begin{cor} \label{c:classification2}
	Let $(M^4,g,I,J)$ be a compact gradient steady GKRS. Then precisely one of the following holds
	\begin{enumerate}
		\item $(M,I)$ is K\"ahler, $g$ is a Calabi-Yau metric, i.e., $\mathrm{Rc}_g=0$, and $I=\pm J$.
		\item $(M^4,g,I,J)$ is an odd-type GK structure, with $(M^4,I)$ biholomorphic to a quotient of a diagonal Hopf surface $(\til M,I_{\alpha\beta})$ by $\Gamma_{k,\ell}\simeq \Z_{\ell}$.
		Up to the action of $\mathrm{Aut}(M,I)$ and scaling, the metric $g$ and the complex structure $J$ are given by $g^s$ and $\pm J_{\mathrm{odd}}^s$ of Theorem~\ref{thm:hopf_soliton_existence}.
		\item $(M,g,I,J)$ is an even-type GK structure, with $(M,I)$ biholomorphic to a quotient of a diagonal Hopf surface $(\til M,I_{\alpha\beta})$ by $\Z_{\ell}$ acting via multiplication by $\Gamma_{k,\ell}\simeq \Z_{\ell}$.
		Up to the action of $\mathrm{Aut}(M,I)$ and scaling, the metric $g$ and the complex structure $J$ are given by $g^s$ and $\pm J_{\mathrm{even}}^s$ of Theorem~\ref{thm:hopf_soliton_existence}.
	\end{enumerate}
\end{cor}

\begin{proof}
	Assume that a compact complex surface $(M,I)$ is a part of a gradient steady GKRS $(M,g,I,J)$. It follows from~\cite[Theorem\,1.1]{st-19-soliton} (see also~\cite{st-us-19}) that either $(M,g,I)$ is K\"ahler, Calabi-Yau, or $(M,I)$ is biholomorphic to a Hopf surface. 
	In the latter case, by the results~\cite{ap-gu-07} and Theorem~\ref{thm:hopf_existence_even} we know that regardless of whether $(M,g,I,J)$ is even or odd, $(M,I)$ must be biholomorphic to the quotient of a primary Hopf surface $(\til M,I_{\alpha\beta})=(\C^2\backslash\{\mathbf 0\})/\la A\ra$ by a cyclic group $\Gamma_{k,\ell}\simeq \Z_\ell$ acting via multiplication by primitive $\ell^{\mathrm{th}}$ roots of 1.  Now, assuming that $(M,I)$ is such Hopf surface, we prove the uniqueness of odd/even gradient steady solitons up to the action of $\mathrm{Aut}(M,I)$. Below we consider the case of even GK structures, odd being analogous.  Assume that $(M,I)$ is a $\Gamma_{k,\ell}\simeq \Z_\ell$-quotient of a diagonal Hopf surface $(\til M,I_{\alpha\beta})$ (see Definition~\ref{def:hopf}). We will need the following lemma extending Proposition~\ref{prop:hopf_Q_sigma}.
	\begin{lemma}\label{lm:aut_secondary}
		Let $(M,I)=(\til M,I_{\alpha\beta})/\Gamma_{k,\ell}$ be a cyclic quotient of a diagonal Hopf surface. Then $\mathrm{Aut}(M,I)$ acts transitively on the disconnected canonical divisors on $(M,I)$.
	\end{lemma}
	\begin{proof}
		Let $\pi\colon (\til M,I_{\alpha\beta})\to (M,I)$ be the natural projection. Denote by $D$ a disconnected canonical divisor on $(M,I)$. Then $E=\pi^{-1}D$ is a disconnected $\Gamma_{k,\ell}$-invariant canonical divisor on $(\til M,I_{\alpha\beta})$. An automorphism $\gamma\in\mathrm{Aut}(\til M,I_{\alpha\beta})$ descends to an automorphism of the $\Gamma_{k,\ell}$-quotient $(M,I)$ if and only if $\gamma$ belongs to the normalizer of $\Gamma_{k,\ell}$ in $\mathrm{Aut}(\til M,I_{\alpha\beta})$:
		\[
		\gamma\in N_{\mathrm{Aut}(\til M,I_{\alpha\beta})}(\Gamma_{k,\ell}).
		\]
		Thus the statement of the lemma would follow from the following:
		\medskip
		
		\noindent\textbf{Claim.} $N_{\mathrm{Aut}(\til M,I_{\alpha\beta})}(\Gamma_{k,\ell})$ acts transitively on the set of $\Gamma_{k,l}$-invariant disconnected canonical divisors on $(\til M,I_{\alpha\beta})$.
		\medskip
		
		To prove this claim we have to go through various cases on $(\alpha,\beta)$ using Proposition~\ref{prop:hopf_aut}:
		
		\begin{enumerate}
			\item If $\alpha=\beta$, then $\mathrm{Aut}(\til M,I_{\alpha\beta})=\mathrm{GL}_2(\C)/\la A\ra$, where $A=\left(\begin{matrix}
			\alpha & 0 \\ 0 & \beta
			\end{matrix}\right)$, and any disconnected canonical divisor is given by the zero set of a nondegenerate quadratic form $q(z_1,z_2)$.  If $k\neq 1$ in $\Gamma_{k,\ell}$, then the only $\Gamma_{k,\ell}$-invariant canonical divisor on $(\til M,I_{\alpha\beta})$ is the standard one $(\{z_1=0\}\cup \{z_2=0\})/\la A\ra$ and the claim is tautologically true.  If $k=1$, then $\Gamma_{k,\ell}$ lies in the center of $\mathrm{Aut}(\til M,I_{\alpha\beta})$ so that the normalizer of $\Gamma_{k,\ell}$ is the entire automorphism group which acts transitively on the set of all disconnected canonical divisors by Proposition \ref{prop:hopf_Q_sigma}.
			
			\item If $\alpha=\beta^q$, then $\mathrm{Aut}(\til M,I_{\alpha\beta})=G/\la A\ra$, where group $G\simeq (\C^*)^2\rtimes\C$ acts on $\C^2\backslash\{\mathbf{0}\}$ as
			\[
			(z_1,z_2)\mapsto (az_1+bz_2^q,dz_2),\quad a,d\in\C,^*\ b\in\C,
			\]
			If $k\neq q$, then the only disconnected canonical divisor is the standard one, and the claim is again tautologically true.  If $k\neq q$, then $\Gamma_{k,\ell}$ lies in the center of $\mathrm{Aut}(\til M,I_{\alpha\beta})$. As before it means that the normalizer of $\Gamma_{k,\ell}$ coincides with the entire automorphism group which acts transitively on the set of all disconnected canonical divisors by Proposition \ref{prop:hopf_Q_sigma}.
			
		\item If $\alpha\neq \beta^q$ and $\beta\neq \alpha^q$ for any $q$, then there is only one disconnected canonical divisor, and the claim again trivially holds.
		\end{enumerate}
		With the claim proved we conclude the proof of Lemma~\ref{lm:aut_secondary}.
	\end{proof}
	
	We now get back to the proof of the corollary. Let $(M,g_1,I,J_1)$ and $(M,g_2,I,J_2)$ be two gradient steady even GKRS on a secondary Hopf surface $(M,I)=(\til M,I_{\alpha\beta})/\Gamma_{k,\ell}$. Each GK structure defines a Poisson tensor
	\[
	\sigma_1=\frac{1}{2}[I,J_1]g_1^{-1}\quad \sigma_2=\frac{1}{2}[I,J_2]g_2^{-1}.
	\]
	By Lemma~\ref{lm:aut_secondary}, we can find $\gamma\in\mathrm{Aut}(M,I)$ such that $\sigma_1$ and $\gamma^*\sigma_2$ have a common zero locus.
	
	The GK structures $(g_1,I,J_1)$ and $(\gamma^*g_2,I,\gamma^*J_2)$ lift to GK solitons on the primary Hopf surface $(\til M,I_{\alpha\beta})$, which we denote by the same symbols. By Theorem~\ref{thm:hopf_soliton_uniqueness}, there exists $\phi\in \mathrm{Aut}(\til M,I_{\alpha\beta})$ such that
	\[
	\phi^*(\til M,g_1,I_{\alpha\beta},J_1)=(\til M,\gamma^*g_1,I_{\alpha\beta},\gamma^*J_1)
	\]
	On the other hand both GK structures $(\til M,g_1,I_{\alpha\beta},J_1)$
	$(\til M,\gamma^*g_1,I_{\alpha\beta},\gamma^*J_1)$ share a common zero locus of the underlying Poisson tensor, thus by Proposition~\ref{prop:hopf_Q_sigma}, $\phi\in (\C^*)^2/\la A\ra$. Since $\Gamma_{k,\ell}$ is also a subgroup of $(\C^*)^2/\la A\ra$ and this group is abelian, $\phi$ descends to an automorphism of $(M,I)=(\til M,I_{\alpha\beta})/\Gamma_{k\ell}$.
\end{proof}

\section{The space of generalized K\"ahler structures} \label{s:SOGKS}

In K\"ahler geometry, the space of K\"ahler metrics in a given K\"ahler class forms on open cone.  Furthermore, the space of K\"ahler classes is characterized using the result of Demailly-Paun \cite{DemPaun}, thus the space of K\"ahler metrics on a given K\"ahler manifold is fairly well-understood, allowing for the definition of functionals characterizing interesting canonical representatives.  As we have seen above, in the setting of generalized K\"ahler geometry, the questions of the global structure of the space of GK metrics, the underlying Poisson geometry, the definition of interesting functionals, and the existence of canonical metrics, are all linked.  In this section we explore these questions in a more general way.

We first note that given a generalized K\"ahler manifold $(M,g,I,J)$, we can `forget' one of the complex structures and obtain two underlying Hermitian manifold $(M,g,I)$ and $(M,g,J)$, which are automatically \emph{pluriclosed}, i.e.,  $dd^c_I\omega_I=dd^c_J\omega_J=0$. A fundamental question is whether this procedure can be inverted (see also~\cite[Question\,1]{ap-gu-07}):

\begin{question}\label{q:complex_to_gk}
Given a compact complex manifold $(M,I)$ admitting some pluriclosed metric $g_0$, when is it possible to upgrade $(M,I)$ to a generalized K\"ahler structure $(M,g,I,J)$ with a possibly different metric?
\end{question}

Clearly, the answer to the question is positive if $(M,I)$ admits a K\"ahler metric $g$~--- in this case one can take $J=I$ or $J=-I$. The situation is much more subtle if $(M,I)$ does not admit any K\"ahler metric. This question happens to be particularly fruitful if $\dim_\C M=2$, since any compact complex surface $(M,I)$ admits a unique up to a constant factor pluriclosed metric in every conformal class. In this setting, generalized K\"ahler structures appeared among bihermitian structures, and their existence and classification were actively studied in the connection to the real conformal 4-dimensional geometry, cf.\,\cite{ap-ga-gr-99,ap-gu-07,ap-ba-dl-17, fu-po-10,fu-po-14, fu-po-19, sa-91, st-us-19}.  

Given that any GK manifold $(M,g,I,J)$ has an underlying holomorphic Poisson tensor $\sigma_I$, it is natural to formulate a refinement of Question~\ref{q:complex_to_gk}.
\begin{question}\label{q:poisson_to_gk}
	Let $(M,I,\pmb\pi)$ be a compact complex manifold with a holomorphic Poisson tensor $\pmb\pi\in\Gamma(\Lambda^{2}(T^{1,0}M))$. Assume that $(M,I)$ admits a pluriclosed metric. Is it possible to upgrade $(M,I)$ to a generalized K\"ahler structure $(M,g,I,J)$ so that the Poisson tensor underlying $(M,g,I,J)$ equals the given one: $\sigma_I=\pmb\pi$?
\end{question}

There is an abundance of examples, where the answer to the above question is positive. For instance, the deformation theory developed by Goto~\cite{go-10,go-12}, implies that given a compact K\"ahler surface $(M,I)$, and a holomorphic Poisson tensor $\pmb\pi$, there exist GK structures $(M,g,I,J)$ with the underlying Poisson tensor $\sigma_I=\pmb\pi$. At the same time,  there are some rather subtle obstructions in the non-K\"ahler case. For example, a non-diagonal primary Hopf surface (see Definition~\ref{def:hopf}) has a unique and connected anti-canonical divisor, but it does not  admit GK structures, as a consequence of \cite[Th.~3]{ap-gu-07} and  \cite[Prop. 4]{ap-ga-gr-99}.  Furthermore,  we showed in Theorem \ref{thm:hopf_poisson_gk} that not \emph{all} Poisson tensors in the  pencil determined by an anti-canonical divisor $D$ consisting of two elliptic curves can actually be realized by a GK structure on a type $1$ Hopf surface.

The above questions are mostly aimed at understanding the possible space of generalized K\"ahler classes, as defined through the flow construction in \cite{gi-st-20}.  Going further, given such a fixed generalized K\"ahler class, it is natural to ask for its topology:
\begin{question}\label{q:topologyM}
	Let $(M,I,J)$ be a compact generalized K\"ahler manifold.  What is the topology of the associated generalized K\"ahler class $\mathcal M$?
\end{question}
A natural approach to this question is through the use of \emph{generalized K\"ahler-Ricci flow} (GKRF) \cite{st-ti-12}, a special case of pluriclosed flow \cite{st-ti-10}.  GKRF is a natural extension of K\"ahler-Ricci flow to the setting of GK geometry, which when $c_1(M, I) = 0$ will evolve via the flow construction (albeit with a singular Hamiltonian potential function).  For instance, global existence and convergence results for GKRF (\cite{ap-st-17, ga-jo-st-21, st-bi, st-16, st-17}) yield the topology of generalized K\"ahler structures in certain settings (see also Remark \ref{r:flowrmk}).

\end{document}